%% file: main.tex
\documentclass{amsart}
\input{top}

\title{Improved algebraic fibrings}
\author{Sam P.~Fisher}
\email{sam.fisher@maths.ox.ac.uk}
\address{University of Oxford, United Kingdom}


\begin{document}

\begin{abstract}
    We show that a virtually RFRS group $G$ of type $\mathtt{FP}_n(\Q)$ virtually algebraically fibres with kernel of type $\mathtt{FP}_n(\Q)$ if and only if the first $n$ $\ell^2$-Betti  numbers of $G$ vanish, that is, $b_p^{(2)}(G) = 0$ for $0 \leqslant p \leqslant n$. This confirms a conjecture of Kielak. We also offer a variant of this result over other fields, in particular in positive characteristic.

    As an application of the main result, we show that amenable virtually RFRS groups of type $\mathtt{FP}(\Q)$ are virtually Abelian. It then follows that if $G$ is a virtually RFRS group of type $\mathtt{FP}(\Q)$ such that $\Z G$ is Noetherian, then $G$ is virtually Abelian. This confirms a conjecture of Baer for the class of virtually RFRS groups of type $\mathtt{FP}(\Q)$, which includes (for instance) the class of virtually compact special groups.
\end{abstract}

\maketitle


\section{Introduction}

A group $G$ is \textit{algebraically fibred} (or simply \textit{fibred}) if it admits a homomorphism onto $\Z$ with a finitely generated kernel. The interest in algebraic fibrings arose from the study of $3$-manifolds fibring over the circle. Using the long exact sequence of homotopy groups associated to a fibration, one sees that a surface bundle $M \rightarrow S^1$, where $M$ is a compact $3$-manifold, induces an algebraic fibration $\pi_1(M) \rightarrow \Z$. A celebrated theorem of Stallings \cite{Stallings3mflds} establishes the converse: if $G$ is isomorphic to the fundamental group of a closed, compact $3$-manifold $M$, then an algebraic fibration $G \rightarrow \Z$ is induced by a surface bundle $M \rightarrow S^1$.

Recently, Kielak characterised the virtual algebraic fibring of residually finite rationally solvable (RFRS) groups in terms of the vanishing of the first $\ell^2$-Betti number. More precisely, he showed the following.

\begin{thm}[Kielak \cite{KielakRFRS}] \label{thm:kielak}
    Let $G$ be an infinite finitely generated virtually RFRS group. Then $G$ virtually algebraically fibres if and only if $b_1^{(2)}(G) = 0$.
\end{thm}

Virtually RFRS groups arise naturally in geometric group theory; for example subgroups of right-angled Artin groups and right-angled Coxeter groups are virtually RFRS and, in particular, special groups (in the sense of Haglund and Wise) are RFRS. Moreover, the RFRS property passes to subgroups and is preserved by taking products and free products of RFRS groups. Kielak's theorem is an algebraic analogue of the following theorem of Agol, which was a key step in confirming Thurston's Virtually Fibred Conjecture. 

\begin{thm}[Agol \cite{AgolCritVirtFib}]
    Every compact and irreducible $3$-manifold $M$ with $\chi(M) = 0$ and nontrivial RFRS fundamental group admits a finite covering that fibres over the circle.
\end{thm}

Since algebraic fibrings induce topological fibrings of $3$-manifolds, Kielak's theorem generalises the above theorem of Agol by removing the assumption that $G$ is the fundamental group of a $3$-manifold. Note that \cite[Theorem 4.1]{Luck02} states that if $M$ is a compact, irreducible $3$-manifold with no $S^2$ boundary components, then $b_1^{(2)}(\pi_1(M)) = -\chi(M)$. Thus, we interpret the condition $b_1^{(2)}(G) = 0$ in Kielak's theorem as the algebraic analogue of the condition $\chi(M) = 0$ in Agol's theorem.

In light of Kielak's theorem, it is natural to ask whether the vanishing of higher $\ell^2$-Betti numbers of a group $G$ controls the higher finiteness properties of the kernel of the virtual fibration. Indeed, Kielak conjectured that a virtually RFRS group of type $\mathtt{FP}_n(\mathbb{Q})$ virtually algebraically fibres with kernel of type $\mathtt{FP}_n(\Q)$ if and only if $b_p^{(2)}(G)$ vanishes for all $p \leqslant n$ \cite[Conjecture 8]{KielakOber20}. The main result of this paper confirms Kielak's conjecture, and gives another characterisation of RFRS groups virtually fibring with kernel of type $\FP_n(\Q)$. We also note that the equivalence of (2) and (3) in the following theorem generalises \cite[Corollary 1.5]{JaikinZapirain2020THEUO}, where Jaikin-Zapirain proves the $n = 1$ case.

\begin{manualtheorem}{A}\label{thm:A}
Let $G$ be a virtually RFRS group of type $\mathtt{FP}_n(\Q)$. Then the following are equivalent:
    \begin{enumerate}[label = (\arabic*)]
        \item\label{item:b2vanish} $b_p^{(2)}(G) = 0$ for all $p \leqslant n$;
        \item\label{item:FPn} there is a finite-index subgroup $H \leqslant G$ and a surjection $\varphi \colon H \rightarrow \Z$ such that $\ker \varphi$ is of type $\FP_n(\Q)$;
        \item\label{item:finiteBetti} there is a finite-index subgroup $H' \leqslant G$ and a surjection $\varphi' \colon H' \rightarrow \Z$ such that $b_p(\ker \varphi') < \infty$ for all $p \leqslant n$.
    \end{enumerate}
\end{manualtheorem}

It should be emphasized that if $\ker \varphi'$ has finite Betti numbers in dimensions $\leqslant n$, it is not necessarily the case that $\ker \varphi'$ is of type $\FP_n(\Q)$. We prove a more general theorem that treats algebraic fibring with kernels of type $\FP_n(\mathbb F)$ for any skew-field $\mathbb F$, from which we obtain \cref{thm:A} as a special case.  Before stating the result, we give some background. If $\mathbb{F}$ is a skew-field and $G$ is a locally indicable group, then under certain conditions the group ring $\mathbb{F}G$ embeds into a skew-field $\mathcal{D}_{\mathbb{F}G}$ called the \textit{Hughes-free division ring} of $\mathbb{F}G$ (see \cref{def:HfreeDiv}). If it exists, $\mathcal{D}_{\mathbb{F}G}$ is unique up to $\mathbb{F}G$-algebra isomorphism \cite{HughesDivRings1970} and the \textit{$\mathcal{D}_{\mathbb{F}G}$-homology} of $G$ in dimension $p$ is defined to be $H_p^{\DF{G}}(G) := H_p(G; \DF{G})$. The $p$th \textit{$\mathcal{D}_{\mathbb{F}G}$-Betti number} is defined to be
\[
    b_p^{\mathcal{D}_{\mathbb{F}G}}(G) := \dim_{\mathcal{D}_{\mathbb{F}G}} H_p^{\mathcal{D}_{\mathbb{F}G}} (G).
\]

In \cite[Corollary 1.3]{JaikinZapirain2020THEUO}, Jaikin-Zapirain proves that if $G$ is finitely generated and RFRS, then $\mathcal{D}_{\mathbb{F}G}$ exists for any skew-field $\mathbb{F}$ and, if $\mathbb{F} = \mathbb{Q}$, it is isomorphic to the Linnell skew-field of $G$. For the purposes of this paper, it will not be necessary to know how the Linnell skew-field is defined, but that it can be used to define the $\ell^2$-homology and $\ell^2$-Betti numbers of a group $G$ (see \cref{def:l2b}). Importantly for us, when $G$ is finitely generated and RFRS, we have $b_p^{\mathcal{D}_{\Q G}} (G) = b_p^{(2)}(G)$ for all $p$. We prove the following theorem, which reduces to \cref{thm:A} in the case $\mathbb F = \Q$.

\begin{manualtheorem}{B}\label{thm:B}
    Let $\mathbb{F}$ be a skew-field and let $G$ be a virtually RFRS group of type $\FP_n(\mathbb F)$. Then the following are equivalent:
    \begin{enumerate}
        \item $b_p^{\DF{G}} = 0$ for all $p \leqslant n$;
        \item there is a finite-index subgroup $H \leqslant G$ and a surjection $\varphi \colon H \rightarrow \Z$ such that $\ker \varphi$ is of type $\FP_n(\mathbb F)$;
        \item there is a finite index subgroup $H' \leqslant G$ and a surjection $\varphi' \colon H' \rightarrow \Z$ such that $b_p(\ker \varphi'; \mathbb F) < \infty$ for all $p \leqslant n$.
    \end{enumerate}
\end{manualtheorem}

We highlight the following corollary, which implies, in particular, that if $\mathbb F$ and $\mathbb F'$ are skew-fields of the same characteristic, then a RFRS group $G$ is $\DF{G}$-acyclic in dimension $\leqslant n$ if and only if it is $\mathcal D_{\mathbb F'G}$-acyclic in dimensions $\leqslant n$. Moreover, it also implies that if $G$ is $\mathcal D_{\mathbb F_p G}$ acyclic in dimensions $\leqslant n$ for some prime $p$, then it is also $\ell^2$-acyclic in dimensions $\leqslant n$.

\begin{manualcor}{C}[\cref{cor:charac}]
    Let $G$ be a virtually RFRS group and let $n \in \N$.
    \begin{enumerate}
        \item If $\mathbb F$ and $\mathbb F'$ are skew-fields of the same characteristic, then $G$ virtually algebraically fibres with kernel of type $\FP_n(\mathbb F)$ if and only if it virtually algebraically fibres with kernel of type $\FP_n(\mathbb F')$.
        \item If $p$ is a prime such that $G$ virtually algebraically fibres with kernel of type $\FP_n(\mathbb F_p)$, then it virtually fibres with kernel of type $\FP_n(\Q)$.
    \end{enumerate}
\end{manualcor}

The final section of the paper is devoted to some applications of the main theorems. An interesting question is to determine general conditions under which amenable groups are elementary amenable. There are many examples of amenable groups that are not elementary amenable, for instance, Grigorchuk's group of intermediate growth, but it is not known whether there are examples of amenable groups of finite cohomological dimension that are not elementary amenable. Moreover, elementary amenable groups of finite cohomological dimension are virtually solvable by \cite[Lemma 2]{Hillman91} and \cite[Corollary 1]{HillmanLinnell}. This leads us to the following question, which first appeared in the work of Degrijse.

\begin{q}[Degrijse \cite{degrijse2016amenable}]
Are amenable groups of finite cohomological dimension over $\Z$ virtually solvable?
\end{q}

We obtain the following as an application of \cref{thm:A}, which provides evidence of a positive answer for virtually RFRS groups.

\begin{manualtheorem}{D}[\cref{thm:amRFRSelemAm}]\label{thm:C}
    If $G$ is a virtually (amenable RFRS group of type $\mathtt{FP}(\Q)$), then $G$ is virtually Abelian.
\end{manualtheorem}

We think of this result as being in the same vein as the well-known fact that nilpotent RFRS groups must be virtually Abelian (see, e.g.~\cite[Proposition 2.5]{KoberdaSuciu_RFRp}). As a corollary, we confirm the following conjecture of Baer for virtually RFRS groups of type $\mathtt{FP}(\Q)$. Note that this class of groups includes all virtually compact special groups.

\begin{conj}
    If $G$ is a group such that the group ring $\Z G$ is Noetherian, then $G$ is polycyclic-by-finite.
\end{conj}

Hall showed that polycyclic-by-finite groups have Noetherian group rings \cite[Theorem 4]{Hall59}, but it is still unknown whether the converse holds. Some progress was made recently by P.~Kropholler and Lorensen, who showed that if $RG$ is right Noetherian and $R$ is a domain, then $G$ is amenable and all of its subgroups are finitely generated \cite[Corollary B]{KrophollerLorensen19}. This result provides evidence for the conjecture, as the only known amenable groups in which every subgroup is finitely generated are polycyclic-by-finite. We obtain the following as a consequence of \cref{thm:C} and Kielak's appendix to \cite{BartholdiKielakApp}.

\begin{manualcor}{E}[\cref{cor:baer}]
    Let $G$ be a virtually RFRS group of type $\mathtt{FP}(\Q)$ such that $\Z G$ is Noetherian. Then $G$ is virtually Abelian.
\end{manualcor}

Finally, we mention that Llosa Isenrich, Martelli, and Py remarked in  \cite{IsenrichPyMartelli_F3notF4} that an easy consequence of \cref{thm:A}, Agol's Theorem \cite{AgolHaken}, and work of Agol and Bergeron--Haglund--Wise \cite{BHW_2011} is the existence of hyperbolic groups containing type $\FP_{n-1}(\Q)$ subgroups that are not of type $\FP_n(\Q)$ for all $n \in \N$. We also remark that, by the same argument, such lattices in $\PO(2n+1,1)$ algebraically fibre with kernel of type $\FP(\Q)$.

\begin{prop}[{\cite[Proposition 19]{IsenrichPyMartelli_F3notF4}}]\label{prop:LMP}
    Let $\Gamma \in \PO(m,1)$ be a hyperbolic arithmetic lattice of the simplest type. If $m = 2n$, then $\Gamma$ virtually fibres with kernel of type $\FP_{n-1}(\Q)$ but not of type $\FP_n(\Q)$. If $m = 2n+1$, then $\Gamma$ virtually fibres with kernel of type $\FP(\Q)$.
\end{prop}

Our methods do not allow us to say anything about algebraic fibring with kernels of type $\FP_n(\mathbb F_p)$ for $p$ prime (nor about algebraic fibring with stronger finiteness properties of the kernel) since the $\mathcal D_{\mathbb F_p G}$-Betti numbers of simplest type hyperbolic arithmetic lattices are not known to vanish. In a subsequent paper \cite{IsenrichPy2022}, Llosa Isenrich and Py showed that there are hyperbolic arithmetic lattices of the simplest type in $\PU(n,1)$ that virtually fibre with kernel of type $\F_{n-1}$ but not of type $\FP_n(\Q)$, answering a question of Brady about the existence of subgroups of hyperbolic groups with exotic finiteness properties.

\subsection*{Structure of the paper} 

In \cref{sec:prelims} we introduce some of the main tools and objects that will be used throughout the paper. In particular, we define finiteness properties of groups, Hughes-free division rings, RFRS groups, and Ore localisation.

\cref{sec:valuations} recalls what we will need from the theory of valuations on free resolutions developed by Bieri and Renz in \cite{BieriRenzValutations}. The results in \cite{BieriRenzValutations} are stated for free resolutions over group rings with coefficients in $\Z$, however we will need them for coefficients in an arbitrary associative, unital ring $R$. There is no essential dependence on the ring, so our proofs are similar to those of Bieri and Renz after replacing $\Z$ with $R$.

In \cref{sec:horo}, we introduce the complex of horochains associated to a free resolution equipped with a valuation.

\cref{sec:SigInv} begins with the definition of the higher $\Sigma$-invariants $\Sigma_R^n(G; M)$ for a group $G$, a unital, associative ring $R$, and an $RG$-module $M$. Again, these were introduced in the case $R = \Z$ in \cite{BieriRenzValutations}, and reduce to the usual Bieri--Neumann--Strebel invariant when $n = 1$, $R = \Z$, and $M = \Z$ is the trivial $\Z G$-module. The rest of the section is devoted to the proof of \cref{thm:Main}, which gives equivalent characterisations of the invariants $\Sigma_R^n(G; M)$. The most important characterisation for our purposes is the following: $[\chi] \in \Sigma_R^n(G;M)$ if and only if $\Tor_i^{RG}(\widehat{RG}^\chi, M) = 0$ for all $i \leqslant n$, where $\widehat{RG}^\chi$ is the Novikov ring. When $n = 1$, this result is Sikorav's theorem \cite{SikoravThese} and is a key ingredient in Kielak's proof of \cref{thm:kielak}. The proof follows arguments given in \cite{BieriRenzValutations} and Schweitzer's appendix to \cite{BieriDeficiency}. \cref{thm:Main} is the main technical result that will be used in the proof of \cref{thm:agrarianMain}. 

In \cref{sec:homology} we introduce $\mathcal{D}_{\mathbb{F}G}$-homology and prove properties of $\mathcal{D}_{\mathbb{F}G}$-Betti numbers analogous to those of $\ell^2$-Betti numbers. Namely, we prove that 
\[
    [G : H] \cdot b_p^{\mathcal{D}_{\mathbb{F}G}}(G) = b_p^{\mathcal{D}_{\mathbb{F}H}}(H)
\]
(\cref{lem:JBscales}) whenever $\mathcal{D}_{\mathbb{F}G}$ exists and $H$ is a finite index subgroup of $G$. In \cref{thm:JBSES}, we show that if $b_p^{\mathcal{D}_{\mathbb{F}K}}(K) < \infty$ and $G$ fits into a short exact sequence $1 \rightarrow K \rightarrow G \rightarrow \Z \rightarrow 1$, then $b_p^{\mathcal{D}_{\mathbb{F}G}}(G) = 0$. This should be thought of as an analogue of \cite[Theorem 7.2]{Luck02} for $\mathcal{D}_{\mathbb{F}G}$-Betti numbers. We then prove the main result, \cref{thm:agrarianMain}, and obtain \cref{thm:b2rfrs} as a special case. In \cref{thm:finiteBetti}, we show that virtually fibring with kernel of type $\FP_n(\Q)$ is equivalent to having a virtual map to $\Z$ whose kernel has finite Betti numbers in dimensions $\leqslant n$.

We conclude with the proofs of \cref{thm:amRFRSelemAm} and \cref{cor:baer} in \cref{sec:app}, and mention related work of Llosa Isenrich, Martelli, and Py.

\subsection*{Acknowledgements.} The author would like to thank Dawid Kielak for numerous helpful conversations and comments on this paper, Peter Kropholler for a helpful correspondence, and Sam Hughes for pointing out the application of \cref{thm:amRFRSelemAm} to \cref{cor:baer}. The author also thanks Sami Douba for pointing out that a polycyclic RFRS group must in fact be Abelian.

This work has received funding from the European Research Council (ERC) under the European Union's Horizon 2020 research and innovation programme (Grant agreement No. 850930).

\section{Preliminaries} \label{sec:prelims}

\begin{rem*}
In the sequel, all rings will be associative and unital with $1 \neq 0$.
\end{rem*}

\subsection{Finiteness properties}

\begin{defn}[type $\mathtt{FP}_n$]
Let $R$ be a ring and $M$ be a left $R$-module. We say that $M$ is of \textit{type $\mathtt{FP}_n$}, and write $M \in \mathtt{FP}_n$, if $M$ has a projective resolution
\[
\cdots \rightarrow P_{n+1} \rightarrow P_n \rightarrow \cdots \rightarrow P_1 \rightarrow P_0 \rightarrow M \rightarrow 0
\]
by left $R$-modules, where $P_j$ is finitely generated for $j \leqslant n$. If we want to specify the ring, we say that $M$ is of \textit{type $\mathtt{FP}_n$ over $R$} and write $M \in \mathtt{FP}_n(R)$.  If $P_j = 0$ for $j > n$, then we write $M \in \mathtt{FP}(R)$.

A group $G$ is of \textit{type $\mathtt{FP}_n$ over $R$} if the trivial $RG$-module $R$ is of type $\mathtt{FP}_n(RG)$; in this case we write $G \in \mathtt{FP}_n(R)$. Similarly, if the trivial $RG$-module $R$ is of type $\mathtt{FP}(RG)$, then we write $G \in \mathtt{FP}(R)$.
\end{defn}

We will often use the fact that an $R$-module $M$ is of type $\mathtt{FP}_n$ if and only if there is a free resolution $\cdots \rightarrow F_{n+1} \rightarrow F_n \rightarrow \cdots \rightarrow F_0 \rightarrow M \rightarrow 0$ with $F_j$ finitely generated for $j \leqslant n$ \cite[Proposition VIII.4.3]{BrownGroupCohomology}. Note that the analogous fact does not hold for type $\mathtt{FP}$.

The following definition will not be needed until \cref{sec:app}.

\begin{defn}[cohomological dimension]
A resolution $\cdots \rightarrow P_1 \rightarrow P_0 \rightarrow M \rightarrow 0$ of an $R$-module $M$ has \textit{length} $n$ if $P_n \neq 0$ and $P_m = 0$ for $m > n$. A group $G$ has \textit{cohomological dimension $n$ over $R$} if $n$ is the shortest length of any projective resolution of the trivial $RG$-module $R$. In this case we will write $\cd_R (G) = n$. If $R$ has no finite-length projective resolution, then $\cd_R(G) = \infty$.
\end{defn}

Note that $G \in \mathtt{FP}(R)$ implies that $G$ has finite cohomological dimension over $R$, but not conversely.

\subsection{Hughes-free division rings}

We are following Jaikin-Zapirain's exposition of this material; in particular, \cref{def:HfreeDiv,def:crossedprod} are taken from \cite{JaikinZapirain2020THEUO}.

\begin{defn}[Hughes-free division rings] \label{def:HfreeDiv}
Let $\mathbb{F}$ and $\mathcal{D}$ be skew-fields, let $G$ be a locally indicable group, and let $\varphi \colon \mathbb{F} G \rightarrow \mathcal{D}$ be a ring homomorphism. Then the pair $(\mathcal{D}, \varphi)$ is \textit{Hughes-free} if
\begin{enumerate}[label=(\arabic*)]
    \item\label{item:epic} $\mathcal{D}$ is the skew-field generated by $\varphi(\mathbb{F} G)$, i.e., if $\mathcal E \subseteq \mathcal D$ is a skew-field such that $\varphi(\mathbb F G) \subseteq \mathcal E$, then $\mathcal E = \mathcal D$;
    \item\label{item:direct} for every nontrivial finitely generated subgroup $H \leqslant G$, every normal subgroup $N \triangleleft H$ such that $H/N \cong \Z$, and every set of elements $h_1, \dots, h_n \in H$  lying in pairwise distinct cosets of $N$, the sum
    \[
    \langle\varphi(\mathbb{F}N)\rangle \cdot \varphi(h_1) + \cdots + \langle\varphi(\mathbb{F}N)\rangle \cdot \varphi(h_n)
    \]
    is direct, where $\langle\varphi(\mathbb{F}N)\rangle$ is the sub-skew-field of $\mathcal{D}$ generated by $\varphi(\mathbb{F}N)$.
\end{enumerate}
\end{defn}

Hughes showed that if such a pair $(\mathcal{D},\varphi)$ exists, then $\mathcal{D}$ is unique up to $\mathbb{F}G$-algebra isomorphism \cite{HughesDivRings1970}. Thus, we denote $\mathcal{D}$ by $\mathcal{D}_{\mathbb{F}G}$. Let $H \leqslant G$ be a subgroup and suppose that $\mathcal{D}_{\mathbb{F}G}$ exists. Then $\mathcal{D}_{\mathbb{F}H}$ exists as well and is equal to $\langle \varphi(\mathbb{F}H) \rangle \subseteq \mathcal{D}_{\mathbb{F}G}$. Hence, we will view $\mathcal{D}_{\mathbb{F}H}$ as a subset of $\mathcal{D}_{\mathbb{F}G}$ whenever $H$ is a subgroup of $G$. Moreover, Gr\"ater showed that $\mathcal{D}_{\mathbb{F}G}$ is \textit{strongly Hughes-free} whenever it exists \cite[Corollary 8.3]{Grater20}, which is to say that condition (2) above can be replaced with the following:
\begin{enumerate}[label = {($2^\prime$)}]
    \item\label{item:2prime} for every nontrivial subgroup $H \leqslant G$, every normal subgroup $N \triangleleft H$, and every set of elements $h_1, \dots, h_n \in H$  lying in pairwise distinct cosets of $N$, the sum
    \[
    \langle\varphi(\mathbb{F}N)\rangle \cdot \varphi(h_1) + \cdots + \langle\varphi(\mathbb{F}N)\rangle \cdot \varphi(h_n)
    \]
    is direct.
\end{enumerate}

\begin{defn}[Crossed products] \label{def:crossedprod}
A ring $R$ is \textit{$G$-graded} if its underlying Abelian group is isomorphic to a direct sum $\bigoplus_{g \in G} R_g$ of Abelian groups $R_g$ and $R_g R_h \subseteq R_{gh}$ for all $g,h \in G$. If $R_g$ contains a unit for each $g \in G$, then we say that $R$ is a \textit{crossed product} of $R_e$ and $G$, and we write $R = R_e * G$.
\end{defn}

Strong Hughes-freeness implies the following useful properties, which are stated in \cite{JaikinZapirain2020THEUO}.

\begin{prop}\label{prop:twistedNormalSkew}
Let $G$ be a locally indicable group such that $\mathcal{D}_{\mathbb{F}G}$ exists, let $\varphi$ be as in \cref{def:HfreeDiv}, and $N \triangleleft G$ be a normal subgroup. If $R$ is the subring of $\mathcal{D}_{\mathbb{F}G}$ generated by $\mathcal{D}_{\mathbb{F}N}$ and $\varphi(\mathbb{F}G)$, then
\begin{enumerate}[label = (\arabic*)]
    \item\label{item:subring} $R \cong \mathcal{D}_{\mathbb{F}N} * (G/N)$;
    \item\label{item:twistedFiniteIndex} if $[G:N] < \infty$, then $\mathcal{D}_{\mathbb{F}G} = R$.
\end{enumerate}
\end{prop}

\begin{proof}
    Starting with \ref{item:subring}, let $\{t_i\}_{i \in I}$ be a transversal for $N$ in $G$. We claim that every element of $R$ can be written as a finite sum $\sum_i \alpha_i \varphi(t_i)$, where $\alpha_i \in \mathcal{D}_{\mathbb FN}$ for each $i \in I$. This will conclude the proof of \ref{item:subring} by strong Hughes-freeness. Since every element of $G$ is of the  form $n t_i$ for some $n \in N$ and $i \in I$, it is enough to show that $\varphi(t_i) a \varphi(t_i)\inv \in \mathcal D_{\mathbb F N}$ for every $i \in I$ and every $a \in \mathcal D_{\mathbb F N}$. Notice that $\mathcal D_{\mathbb F N} = \bigcup_{j\geqslant 0} \mathcal D_j$ where $\mathcal D_0 = \varphi(\mathbb FN)$ and $\mathcal D_{j+1}$ is generated by $\mathcal D_j$ and the set $\{x\inv : x \in \mathcal D_j \smallsetminus \{0\}\}$. Because $N$ is normal in $G$, it is clear that $\varphi(\mathbb FN)$ is closed under conjugation by $\varphi(G)$. By induction, we see that each $\mathcal D_j$ is closed under conjugation by $\varphi(G)$.

    \smallskip

    To prove \ref{item:twistedFiniteIndex}, we recall the fact that a finite-dimensional algebra over a skew-field with no zero-divisors is a skew-field. Hence, if $[G:N] < \infty$, then $R \cong \mathcal{D}_{\mathbb{F}N} * (G/N)$ is finite-dimensional over $\mathcal{D}_{\mathbb{F}N}$ and is therefore a skew-field. But $\mathcal{D}_{\mathbb{F}G}$ is the smallest skew-field containing $\varphi(\mathbb{F}G)$, so $\mathcal{D}_{\mathbb{F}G} = R$. \qedhere
\end{proof}


\subsection{RFRS groups}

Residually finite rationally solvable (RFRS) groups were defined by Agol in \cite{AgolCritVirtFib} in order to show that certain hyperbolic $3$-manifolds virtually fibre over the circle. Let $G$ be a group and let $G^\mathsf{ab} := G/[G,G]$ be its Abelianisation. Since $G^\mathsf{ab}$ is Abelian, it is canonically a $\Z$-module, so we can form the tensor product $\Q \otimes_\Z G^\mathsf{ab}$, and there is a group homomorphism $G \rightarrow \Q \otimes_\Z G^{\mathsf{ab}}$ sending $g \in G$ to $1 \otimes g[G,G]$.

\begin{defn}\label{def:RFRS}
A group $G$ is \textit{RFRS} if 
\begin{enumerate}[label = (\arabic*)]
    \item there is a chain $G = G_0 \geqslant G_1 \geqslant G_2 \geqslant \cdots$ of finite index normal subgroups of $G$ such that $\bigcap_{i=0}^\infty G_i = \{1\}$
    \item $\ker(G_i \rightarrow \Q \otimes_\Z G_i^{\mathsf{ab}}) \leqslant G_{i+1}$ for every $i \geqslant 0$.
\end{enumerate}
\end{defn}

The following fact that will be used in the proof of \cref{thm:agrarianMain}. 

\begin{prop}\label{prop:RFRStoQ}
    Let $G$ be a finitely generated RFRS group. Then
    \begin{enumerate}[label=(\arabic*)]
        \item\label{item:resplyZ} $G$ is residually poly-$\Z$;
        \item\label{item:v_retracts} $G$ virtually retracts onto all of its finitely generated Abelian subgroups;
        \item\label{item:loc_ind} $G$ is locally indicable.
    \end{enumerate}
\end{prop}
\begin{proof}
    \cref{item:resplyZ} is proven in \cite[Proposition 4.4]{JaikinZapirain2020THEUO}. We now prove \ref{item:v_retracts}. By \cite[Proposition 1.5]{Minasyan_VRprops}, it suffices to show that $G$ virtually retracts onto its cyclic subgroups. Let $G = G_0 \geqslant G_1 \geqslant \cdots$ be a RFRS chain and let $\langle a \rangle$ be a cyclic subgroup of $G$. Let $i$ be the unique integer such that $a \in G_i \smallsetminus G_{i+1}$. It follows that $a$ is not in the kernel of the free Abelianisation map $G_i \rightarrow G_i^{\mathsf{fab}}$. Let $\overline{a}$ denote the image of $a$ in $G_i^\mathsf{fab}$. There is a finite index subgroup $H' \leqslant G_i^\mathsf{fab}$ and a retraction $H' \rightarrow \langle \overline{a} \rangle$. Let $H \leqslant G_i$ be the preimage of $H'$. The composition $H \rightarrow H' \rightarrow \langle \overline{a} \rangle \rightarrow \langle a \rangle$ is the desired retraction. \cref{item:loc_ind} is an immediate consequence of \ref{item:resplyZ}. \qedhere
\end{proof}


\subsection{Ore localisation}

Ore localisation is an analogue of usual localisation for noncommutative rings. Let $R$ be a ring and let $S$ be its set of non-zero-divisors. Then $R$ satisfies the \textit{Ore condition} if for every $r \in R$ and $s \in S$ there  are elements $p,p' \in R$ and $q,q' \in S$ such that 
\[
    qr = ps \ \textnormal{and} \ rq' = sq'.
\]
If  $S = R \setminus \{0\}$ and $R$ satisfies the Ore condition, then it is called an \textit{Ore domain}, and we can form its \textit{Ore localisation} $\Ore(R)$ as follows. Define an equivalence relation $\sim_R$ on $R \times S$ by declaring that $(r,s) \sim_R (r',s')$ if and only if there are elements $a,b \in S$ such that
\[
    ra = r'b \ \textnormal{and} \ sa = s'b.
\]
The equivalence class of $(r,s)$ under $\sim_R$ is denoted $r/s$ and called a \textit{right fraction}. Then $\Ore(R)$ is defined to be the set of right fractions. We can similarly define an equivalence relation $\sim_L$ and define $\Ore(R)$ as a set of left fractions. The Ore condition ensures that these two constructions are isomorphic and indicates how to convert right fractions into left fractions and vice versa. For a detailed construction of $\Ore(R)$ and the definition of addition and multiplication making this set into a ring, we refer the reader to Section 4.4 of Passman's book \cite{PassmanGrpRng}.

The facts about Ore localisation that we will use are summarised in the following proposition.

\begin{prop}\label{prop:OreLoc}
Let $R$ be an Ore domain. Then
\begin{enumerate}[label = (\arabic*)]
    \item $\Ore(R)$ is a skew-field;
    \item for every $r \in R$, we have $r / 1 = 1 \backslash r$ and the map $R \rightarrow \Ore(R), r \mapsto r/1 = 1 \backslash r$ is an injective ring homomorphism;
    \item \cite[Proposition 2.2(2)]{JaikinZapirain2020THEUO} if $G$ is a group and $\mathbb{F}$ is a skew-field such that $\mathcal{D}_{\mathbb{F}G}$ exists and $K$ is a normal subgroup such that $G / K \cong \Z$, then  $\mathcal{D}_{\mathbb{F}G} \cong \Ore(\mathcal{D}_{\mathbb{F}K} * (G/K))$.
\end{enumerate}
\end{prop}

\section{Valuations on free resolutions} \label{sec:valuations}

In this section, we introduce valuations on free resolutions over a group ring. We will be very closely following Bieri and Renz \cite{BieriRenzValutations} where the theory is developed in the case where the ring is $\Z$. Their proofs go through without change when $
\Z$ is replaced by an arbitrary ring $R$.

Let $R$ be a ring, $G$ a group, and $M$ a left $RG$-module. A \textit{free resolution} of $M$ is an exact sequence
\[
    \cdots \xrightarrow{\partial_{n+1}} F_n \xrightarrow{\partial_n} F_{n-1} \xrightarrow{\partial_{n-1}} \cdots \xrightarrow{\partial_1} F_0 \xrightarrow{\partial_0}M \rightarrow 0
\]
of left $RG$-modules, where $F_i$ is free for all $i \geqslant 0$. We will usually omit the subscripts on the boundary maps $\partial_n$ and denote the free resolution by $F_\bullet \rightarrow M \rightarrow 0$. Let $F$ be the free $RG$-module $\bigoplus_{i = 0}^\infty F_i$, and define the \textit{$n$-skeleton} of $F$ to be $F^{(n)} := \bigoplus_{i = 0}^n F_i$. The elements of $F$ are called \textit{chains}, so a chain is not necessarily an element of $F_i$ for any $i$ in our context. Fixing a basis $X_i$ for each $F_i$, we note that $X := \bigcup_{i=0}^\infty X_i$ is a basis for $F$ and $X^{(n)} := \bigcup_{i = 0}^n X_i$ is a basis for $F^{(n)}$. The resolution $F_\bullet \rightarrow M \rightarrow 0$ is \textit{admissible with respect to $X$} if $\partial x \neq 0$ for every $x \in X$. We will always assume that our free resolutions are admissible with respect to the basis we are working with. This is not a strong requirement, since if all boundary maps are nonzero, then $F$ has a  basis with respect to which $F_\bullet \rightarrow M \rightarrow 0$ is admissible; otherwise we can truncate the resolution and choose a basis to obtain an admissible resolution of finite length. We also define the \textit{support} of a chain $c \in F$ (with respect to $X$), denoted $\supp_X(c)$, as follows: every chain $c \in F$ can be written uniquely as $\sum_{g\in G, x \in X} r_{g,x} gx$, where $r_{g,x} \in R$. Then $\supp_X (c) := \{ gx : r_{g,x} \neq 0 \}$; we will usually drop the subscript $X$ when the basis is understood.

Let $\chi \colon G \rightarrow \R$ be a nontrivial \textit{character}, that is, a nonzero group homomorphism from $G$ to the additive group $\R$. This provides the elements of $G$ with a notion of height, which we now extend to the chains of $F$. Let $\R_\infty = \R \cup \{\infty\}$, where $\infty$ is an element such that $t < \infty$ for every $t \in \R$.  We construct a function $v_X \colon F \rightarrow \R_{\infty}$ via the following inductive procedure. For an element $c \in F_0$, define $v_X(c) = \inf\{ \chi(g) : gx \in \supp(c) \}$. Let $n > 0$ and assume that we have defined $v_X$ on $F_{n-1}$. For $x \in X_n$, let $v_X(x) := v_X(\partial x)$. For $c \in F_n$, set $v_X(c) = \inf\{ \chi(g) + v_X(x) : gx \in \supp(c) \}$. For an arbitrary $c \in F$, write $c = \sum_i c_i$, where $c_i \in F_i$, and define $c = \inf_i\{ v_X(c_i) \}$. We are assuming the convention $v_X(0) = \infty$, since $\supp(0) = \varnothing$. The function $v_X$ is called the \textit{valuation extending $\chi$ with respect to $X$}. It is clear from the definition that $v_X(c) = \inf\{ \chi(g) + v_X(x) : gx \in \supp(c) \}$ for any chain $c \in F$. Again, we will usually drop the $X$ in the subscript when the basis is understood.

\begin{prop} \label{PropPropsOfVal}
For the valuation $v_X = v \colon F \rightarrow \R_\infty$ defined above and for any $c, c' \in F$ and $g \in G$, we have
\begin{enumerate}[label=(\arabic*)]
    \item\label{item:sum} $v(c + c') \geqslant \min\{ v(c), v(c') \}$;
    \item\label{item:scalar} $v(c) \leqslant v(rc)$ for all $r \in R$, and $v(c) = v(rc)$ if $r$ is not a zero-divisor;
    \item\label{item:sumequal} if $v(c) \neq v(c')$, then $v(c + c') = \min\{ v(c), v(c') \}$;
    \item\label{item:groupelement} $v(g c) = \chi(g) + v(c)$;
    \item\label{item:infheight} $v(c) = \infty$ if and only if $c = 0$;
    \item\label{item:bdyincrease} if $c \in \bigoplus_{i \geqslant 1} F_i$, then  $v(\partial c) \geqslant v(c)$.
\end{enumerate}
\end{prop}

\begin{proof}
\ref{item:sum} follows from the fact that $\supp(c+c') \subseteq \supp(c) \cup \supp(c')$. The first part of \ref{item:scalar} follows from the fact that $\supp(c) \supseteq \supp(rc)$. If $r$ is not a zero-divisor then $\supp(c) = \supp(rc)$, which yields the second statement of \ref{item:scalar}.

To prove \ref{item:sumequal}, assume without loss of generality that $v(c) < v(c')$. Then,
\[
    v(c) = v((c+c') - c') \geqslant \min\{ v(c+c'), v(-c') \} = \min\{ v(c+c'), v(c') \}.
\]
Since we assumed that $v(c) < v(c')$, the previous line implies that $\min\{ v(c+c'), v(c') \} = v(c + c')$; hence, $v(c) = \min\{ v(c), v(c') \} \geqslant v(c + c')$. But $v(c + c') \geqslant \min\{v(c),v(c')\}$ by \ref{item:sum}, so we obtain \ref{item:sumequal}.

For \ref{item:groupelement}, we have
\begin{align*}
    v(gc) &= \inf\{ \chi(h) + v(x) : hx \in \supp(gc) \} \\
    &= \inf\{ \chi(g (g\inv h)) + v(x) : hx \in g \cdot \supp c \} \\
    &= \chi(g) + \inf\{ \chi(g\inv h) + v(x) : (g\inv h)x \in \supp c \} \\
    &= \chi(g) + v(c).
\end{align*}

For \ref{item:infheight}, we first show that if $c \in F_n \setminus \{0 \}$, then $v(c) < \infty$ by induction on $n$. This is true for $n = 0$ since $\chi(G) \subseteq \R$. Now let $n > 0$. Since $c \neq 0$, there is some element $gx$ in its support, where $g \in G$ and $x \in X$. Then
\[
    v(c) \leqslant v(gx) = \chi(g) + v(x) = \chi(g) + v(\partial x) < \infty
\]
by the inductive hypothesis and by admissibility of $F_\bullet \rightarrow M \rightarrow 0$ with respect to $X$. For a general nonzero element $c \in F$, write $c = \sum_i c_i$ with $c_i \in F_i$. Then $v(c) = \inf_i\{ v(c_i) \} < \infty$ since at least one of the chains $c_i$ is nonzero. Conversely, if $c = 0$, then $v(c) = \infty$, since the infimum of the empty set is $\infty$.

For \ref{item:bdyincrease}, let $c = \sum_{g \in G, x \in X} r_{g,x} gx \in \bigoplus_{i \geqslant 1} F_i$. Then
\begin{align*}
    v(\partial c) &= v\left(\partial\left(\sum_{g \in G, x \in X} r_{g,x} gx\right)\right) \\
    &= v\left( \sum_{g \in G, x \in X} r_{g,x} g \partial x \right) \\
    &\geqslant \inf\left\{ v(r_{g,x} g\partial x) : gx \in \supp(c) \right\} \tag{by \ref{item:sum}} \\
    &\geqslant \inf\left\{ v( g\partial x) : gx \in \supp(c) \right\}  \tag{by \ref{item:scalar}} \\
    &= \inf\{ \chi(g) + v(\partial x) : gx \in \supp(c) \} \\
    &= \inf\{ \chi(g) + v(x) : gx \in \supp(c) \} \\
    &= v(c). \qedhere
\end{align*}
\end{proof}

\begin{defn}[valuation subcomplex and essential acyclicity]
    Given an admissible free resolution $F_\bullet \rightarrow M \rightarrow 0$ over $RG$ (with respect to some fixed basis $X$), a non-trivial character $\chi \colon G \rightarrow \R$, and the valuation $v \colon F \rightarrow \R_\infty$ extending $\chi$, define the \textit{valuation subcomplex} of $F$ with respect to $v$ to be the chain complex $\cdots \rightarrow F_n^v \rightarrow \cdots \rightarrow F_0^v \rightarrow M \rightarrow 0$, where $F_n^v = \{c \in F_n : v(c) \geqslant 0\}$. We denote the valuation subcomplex by $F_\bullet^v \rightarrow M \rightarrow 0$ and let $F^v := \bigoplus_{i = 0}^\infty F_i^v$. \cref{PropPropsOfVal}\ref{item:bdyincrease} ensures that $F_\bullet^v \rightarrow M \rightarrow 0$ is a chain complex of left $RG_\chi$-modules, where $G_\chi$ is the monoid $\{g \in G : \chi(g) \geqslant 0\}$. It is not hard to show that each $F_i^v$ is a free $RG_\chi$-module and has an $RG_\chi$-basis of cardinality $\verti{X_i}$, where $X_i$ is an $RG$-basis for $F_i$.

    The chain  complex $F_\bullet^v \rightarrow M \rightarrow 0$ is \textit{essentially acyclic in dimension $n$} if there is a real number $D \geqslant 0$ such that for every cycle $z \in F_n^v$ there is a $c \in F_{n+1}$ with $\partial c = z$ and $D \geqslant v(z) - v(c)$. We extend the definition of essential acyclicity to dimension $-1$ by declaring that $v(m) = 0$ for all $m \in M \setminus \{0\}$.
\end{defn}

The definition of essential acyclicity in dimension $n$ is equivalent to the following seemingly weaker condition: for every cycle $z \in F_n^v$, there is a $c \in F_{n+1}$ such that $\partial c = z$ and $v(c) \geqslant -D$. To see this, let $z \in F_n^v$ be a cycle. It is easily shown that $v(F) \subseteq \chi(G) \cup \{\infty\}$, so there is a $g \in G$ such that $\chi(g) = v(z)$. Since $g\inv z$ is also in $F_n^v$ with $v(g\inv z) = 0$, there is some $c \in F_{n+1}^v$ such that $\partial c = g\inv z$ and $v(c) \geqslant -D$. Thus, $\partial (gc) = z$, and $D \geqslant v(z) - v(gc)$.

\section{Horochains} \label{sec:horo}

\begin{defn}[complex of horochains and horo-acyclicity]
    Let $F_\bullet \rightarrow M \rightarrow 0$ be an admissible free resolution with respect to some basis $X$, let $\chi \colon G \rightarrow \R$ be a nontrivial character, and let $v \colon F \rightarrow \R_\infty$ be the valuation extending $\chi$. Define $\widehat{F}$ to be the left $RG$-module of chains that are finitely supported below every height. More precisely, $\widehat{F}$ is the $RG$-module of formal sums $\sum_{g \in G, x \in X} r_{g,x} g x$ such that $\{ gx : v(gx) \leqslant t, r_{g,x} \neq 0 \}$ is finite for every $t \in \R$. The elements of $\widehat{F}$ are called \textit{horochains}. If $\hat{c} \in \widehat{F}$, then its \textit{support} is $\supp_X(\hat{c}) := \{ gx : r_{g,x} \neq 0 \}$. Let $\widehat{F}_i \subseteq \widehat{F}$ be the subset of chains with support in $F_i$ and let $\widehat{F}^{(n)} := \bigoplus_{i = 0}^n \widehat{F}_i$. \cref{PropPropsOfVal}\ref{item:bdyincrease} guarantees that $\partial \colon F_n \rightarrow F_{n-1}$ extends to a map $\partial \colon \widehat{F}_n \rightarrow \widehat{F}_{n-1}$ in the obvious way, so we get a complex $\cdots \rightarrow \widehat{F}_n \rightarrow \cdots \rightarrow \widehat{F}_0 \rightarrow 0$. Note that $\widehat{F}$ is not equal to $\bigoplus_{i = 0}^\infty \widehat{F}_i$ since the support of a horochain might intersect infinitely many of the modules $\widehat{F}_i$. A cycle in the chain complex $\widehat{F}$ is called a \textit{horocycle}. We say that $F_\bullet \rightarrow M \rightarrow 0$ is \textit{horo-acyclic} in dimensions $n \geqslant 0$ with respect to $v$ if the chain complex $\cdots \rightarrow \widehat{F}_1 \rightarrow \widehat{F}_0 \rightarrow 0$ is acyclic in dimension $n$.
\end{defn}

We can extend the definition of $v$ to $\widehat{F}$ by defining $v(\hat{c}) := \inf\{ v(gx) : \supp(\hat{c})\}$ for any horochain $\hat{c}$. If $\hat{c} \neq 0$, then $\{ v(gx) : \supp(\hat{c})\}$ is nonempty and attains a minimum because chains are finitely supported below any given height.  Properties \ref{item:sum} through \ref{item:infheight} of \cref{PropPropsOfVal} hold in this setting with the same proofs. 

A version of \cref{PropPropsOfVal}\ref{item:bdyincrease} holds for horochains, namely we have $v(\partial \hat{c}) \geqslant v(\hat{c})$ for all horochains $\hat{c}$, but we need to modify the proof: If $\hat{c} = 0$, then the claim is clear. Otherwise, let $\hat{c} \neq 0$ be a horochain, and let $gx \in \supp(\hat{c})$ be such that $v(gx) = v(\hat{c})$. By the finite version of \ref{item:bdyincrease}, we have that $v(\partial g'x') \geqslant v(g'x') \geqslant v(gx)$ for every $g'x' \in \supp(\hat{c})$. Since every $g''x'' \in \supp(\partial \hat{c})$ is contained in $\supp(\partial g'x')$ for some $g'x' \in \supp(\hat{c})$, we have that $v(g''x'') \geqslant v(\partial g'x') \geqslant v(gx)$ for every $g''x'' \in \supp(\partial \hat{c})$. Thus, $v(\partial \hat{c}) \geqslant v(\hat{c})$.

\smallskip

The following lemma will be used in the proof of \cref{thm:Main}.

\begin{lem} \label{lem:ExtendHom}
Let $F_\bullet \rightarrow M \rightarrow 0$ (resp.~$F_\bullet' \rightarrow M \rightarrow 0$) be a free resolution over $RG$ admissible with respect to a basis $X$ (resp.~$X'$), and let $v$ (resp.~$v'$) be the valuation extending a nontrivial character $\chi\colon G \rightarrow \R$. Suppose that $F^{(n)}$ is finitely generated, and that $\varphi \colon F \rightarrow F'$ is a homomorphism of $RG$-modules. Then
\begin{enumerate}[label=(\arabic*)]
    \item \label{item:extend} $\varphi$ induces a homomorphism of left $RG$-modules given by
    \[
        \widehat{\varphi} \colon \widehat{F}^{(n)} \rightarrow \widehat{F}', \ \sum r_{g,x} gx \mapsto \sum r_{g,x} g \varphi(x) \ \ ;
    \]
    \item \label{item:valIneq} $v'(\widehat{\varphi}(\hat c)) \geqslant v(\hat{c}) + \min_{x \in X^{(n)}} \{ v'(\varphi(x)) - v(x) \}$ for every $\hat{c} \in F^{(n)}$.
\end{enumerate}
\end{lem}

\begin{proof}
For \ref{item:extend}, we need to show that $\widehat{\varphi}(\hat{c})$ is a horochain  for any horochain $\hat{c} \in \widehat{F}^{(n)}$. To this end, let $\hat{c} = \sum r_{g,x} gx$, and note that there are only finitely many elements $x \in X$ such that $gx \in \supp_X(\hat{c})$. If $\widehat{\varphi}(\hat{c})$ is not a horochain, then the set $\{ gx \in \supp_X(\hat{c}) : v'(g\varphi(x)) \leqslant t \}$ is infinite for some $t \in \R$. Since $F^{(n)}$ is finitely generated, there is some fixed $y \in X$ such that $v'(g\varphi(y)) \leqslant t$ and $gy \in \supp_X(\hat{c})$ for infinitely many values of $g \in G$. But then
\begin{align*}
    v(gy) &= \chi(g) + v(y) \\
        &= \chi(g) + v'(\varphi(y)) + v(y) - v'(\varphi(y))\\
        &= v'(g\varphi(y)) + v(y) - v'(\varphi(y)) \\
        &\leqslant t + v(y) - v'(\varphi(y))
\end{align*}
for infinitely many $gy \in \supp_X(\hat{c})$, but $\hat{c}$ is a horochain. 

For \ref{item:valIneq}, write $\hat{c} = \sum_{x\in X^{(n)}} \hat{c}_x$, where $\hat{c}_x = \sum_{g \in G} r_{g,x} g x$. Then
\begin{align*}
    v'(\widehat{\varphi}(\hat{c})) &\geqslant \min_{x \in X^{(n)}}\{ v'(\widehat{\varphi}(\hat{c}_x))\} \\
    &\geqslant \min_{x \in X^{(n)}} \{ \inf\{ v'(g\varphi(x)) : gx \in \supp \hat{c}_x \}   \} \\
    &= \min_{x \in X^{(n)}}\{ \inf\{ v(gx) : gx \in \supp \hat{c}_x \} + v'(\varphi(x)) - v(x)  \} \\
    &= \min_{x \in X^{(n)}} \{ v(\hat{c}_x) + v'(\varphi(x)) - v(x)  \} \\
    &\geqslant \min_{x \in X^{(n)}} \{ v(\hat{c}_x) \} +  \min_{x \in X^{(n)}}\{ v'(\varphi(x)) - v(x) \} \\
    &= v(\hat{c}) + \min_{x \in X^{(n)}} \{ v'(\varphi(x)) - v(x) \}. \qedhere
\end{align*}  \qedhere
\end{proof}

Note that \cref{lem:ExtendHom}\ref{item:valIneq} applies to chains in $F$, since these are just finite horochains. We will use this in the proof \cref{thm:Main}.

\section{Characterisations of the \texorpdfstring{$\Sigma$}{Sigma}-invariant} \label{sec:SigInv}

We introduce the invariants $\Sigma^n_R(G;M)$, which are generalisations of the classical Bieri--Neumann--Strebel invariant \cite{BNSinv87} and its higher dimensional analogues \cite{BieriRenzValutations}. The only difference is that we work over a general ring $R$, while the higher BNS invariants are defined over $\Z$.

Let $G$ be a group. We declare two characters $\chi,\chi' \colon G \rightarrow \R$ to be \textit{equivalent} if $\chi = \alpha \cdot \chi'$ for some $\alpha > 0$ and let $S(G)$ denote the set of equivalence classes of nonzero characters. We call $S(G)$ the \textit{character sphere} of $G$, because it can be given the topology of a sphere when $G$ is finitely generated. 

\begin{defn}[$\Sigma$-invariants]
Let $M$ be an $RG$-module. Then define
\[
\Sigma^n_R(G;M) = \{ [\chi] \in S(G) : M \in \mathtt{FP}_n(RG_\chi)  \},
\]
where $G_\chi = \{ g \in G : \chi(g) \geqslant 0 \}$. Note that $G_\chi = G_{\chi'}$ if $[\chi] = [\chi']$, so $\Sigma^n_R(G;M)$ is well-defined.
\end{defn}

\begin{defn}[Novikov ring]
Let $G$ be a group, let $R$ be a ring, and let $\chi \colon G \rightarrow \R$ be a character. Then the \textit{Novikov ring} $\widehat{RG}^\chi$ is the set of formal sums
\[
\sum_{g \in G} r_g g
\]
such that $\{ g \in G : r_g \neq 0 \ \text{and} \ \varphi(g) \leqslant t \}$ is finite for every $t \in \R$. We give $\widehat{RG}^\chi$ a ring structure by defining $rg + r'g := (r + r')g$ and $rg \cdot r'g' := rr' gg'$ for $r,r' \in R$, $g,g' \in G$, and extending multiplication to all of $\widehat{RG}^\chi$ in the obvious way.
\end{defn}

\cref{thm:Main} gives several characterisations of the $\Sigma$-invariants and is the main technical tool we will need to prove \cref{thm:agrarianMain}. More specifically, we will need the characterisation of $\Sigma_R^n(G;M)$ in terms of the vanishing of Novikov homology; this is the equivalence of \ref{item:SigmaInv} and \ref{item:TorCond} in the following theorem, which should be thought of as a higher dimensional version of Sikorav's theorem \cite{SikoravThese}.

\begin{thm} \label{thm:Main}
    Let $R$ be a ring, let $M$ be a left $RG$-module of type $\mathtt{FP}_n$, and let $\chi \colon G \rightarrow \R$ be a non-trivial character. Let $F_\bullet \rightarrow M \rightarrow 0$ be a free resolution admissible with respect to a basis $X = \bigcup_{i = 0}^\infty X_i$ and with finitely generated $n$-skeleton $F^{(n)}$. Let $v \colon F \rightarrow \R_\infty$ be the valuation extending $\chi$ with respect to $X$. The following are equivalent:
    \begin{enumerate}[label=(\arabic*)]
        \item\label{item:SigmaInv} $[\chi] \in \Sigma_R^n(G;M)$;
        \item\label{item:EssAcyc} $F_\bullet^v \rightarrow M \rightarrow 0$ is essentially acyclic in dimensions $-1, \dots, n-1$;
        \item\label{item:LiftId} there is a chain map $\varphi \colon F \rightarrow F$ lifting the identity $\id_M$ such that $v(\varphi(c)) > v(c)$ for every $c \in F^{(n)}$;
        \item\label{item:HoroAcyc} $F_\bullet \rightarrow M \rightarrow 0$ is horo-acyclic in dimensions $0, \dots, n$ with respect to $v$;
        \item\label{item:TorCond} $\Tor_i^{RG}(\widehat{RG}^\chi,M) = 0$ for all $0 \leqslant i \leqslant n$.
    \end{enumerate}
\end{thm}

The strategy of the proof will be as follows: we begin by proving \ref{item:EssAcyc} $\Rightarrow$ \ref{item:LiftId} $\Rightarrow$ \ref{item:HoroAcyc} $\Rightarrow$ \ref{item:EssAcyc}. This is done by  Schweitzer in the appendix of \cite{BieriDeficiency} in the case $R = \Z$. Once this is done, we prove the equivalence of \ref{item:HoroAcyc} and \ref{item:TorCond}, again following Schweitzer. Finally, we prove the equivalence of \ref{item:SigmaInv} and \ref{item:EssAcyc} following the appendix to Theorem 3.2 in \cite{BieriRenzValutations}, where again this is done in the case $R = \Z$. The proofs below are essentially the same as those given in the references just cited; there is no crucial dependence on the coefficient ring $R$.

\begin{proof}[Proof of \ref{item:EssAcyc} $\Rightarrow$ \ref{item:LiftId}]
Assume that $F_\bullet^v \rightarrow M \rightarrow 0$ is essentially acyclic in dimensions $\leqslant n-1$ and let $D > 0$ be a constant such that for each $k < n$ and every cycle $z \in F_k$, there is a chain $c \in F_{k+1}$ with $\partial c = z$ and $D \geqslant v(z) - v(c)$. 
We will construct a chain map $\varphi \colon F \rightarrow F$ lifting $\id_M$ such that $v(\varphi(c)) > v(c) + (n - k)D$ for every $c \in F^{(k)}$, which implies \ref{item:LiftId}.

We define $\varphi$ on $F^{(k)}$ by induction on $k$. For the base case, let $x \in X_0$ be arbitrary, and fix some $g \in G$ such that $\chi(g) > (n+1)D$. The element $g\inv\partial x \in M$ is a cycle, so there is some $c_x \in F_0$ such that $\partial c_x = g\inv \partial x$ and $D \geqslant v(g\inv \partial x) - v(c_x) = -v(c_x)$, since $v|_{M \setminus\{0\}} = 0$. Define $\varphi$ on $F^{(0)}$ by setting $\varphi(x) = gc_x$ for each $x \in X_0$. It is  clear that $\mathrm{id}_M  \partial = \partial  \varphi$ on $F^{(0)}$. By \cref{lem:ExtendHom}\ref{item:valIneq},
\[
v(\varphi(c)) \geqslant v(c) + \min_{x \in X_0} \{v(\varphi(x)) - v(x)\} > v(c) + nD
\]
for every $c \in F^{(0)}$.

Let $k > 0$ and suppose $\varphi$ is defined on $F^{(k-1)}$ such that it lifts $\id_M$ and $v(\varphi(c)) > v(c) + (n - k + 1)D$ for all $c \in F^{(k-1)}$. Let $x \in X_k$ and note that $\varphi(\partial x)$ is a cycle. By essential acyclicity, there is a chain $d_x \in F_k$ such that $\partial d_x = \varphi(\partial x)$ and $D \geqslant v(\varphi(\partial x)) - v(d_x)$. Define $\varphi$ on $F^{(k)}$ by setting $\varphi(x) = d_x$. Then $\varphi \partial = \partial \varphi$ by construction, and for every $x \in X_k$ we have
\begin{align*}
v(\varphi(x)) - v(x) &= v(d_x) - v(x) \\
    &\geqslant v(\varphi(\partial x)) - v(x) - D \\
    &= v(\varphi(\partial x)) - v(\partial x) - D \\
    &> (n-k)D
\end{align*}
by induction. By \cref{lem:ExtendHom}\ref{item:valIneq}, we have
\[
v(\varphi(c)) \geqslant v(c) + \min_{x \in X_k} \{ v(\varphi(x)) - v(x) \} > v(c) + (n-k)D. \qedhere
\]
\end{proof}

We pause here to prove a lemma that will immediately imply \ref{item:LiftId} $\Rightarrow$ \ref{item:HoroAcyc} and will be useful in the proofs of \ref{item:HoroAcyc} $\Rightarrow$ \ref{item:EssAcyc} and \ref{item:HoroAcyc} $\Rightarrow$ \ref{item:TorCond}. We recall that the maps $\widehat{H}$ and $\widehat{\varphi}$ that appear in the statement of the lemma below are defined in \cref{lem:ExtendHom}.

\begin{lem}\label{lem:Useful}
    With the assumptions of \cref{thm:Main}, let $\varphi \colon F \rightarrow F$ be a chain map lifting $\id_M$ such that $v(\varphi(c)) > v(c)$ for all $c \in F^{(n)}$ and let $H \colon F \rightarrow F$ be a chain homotopy such that $\partial H + H \partial = \id_{F} - \varphi$. Let $\hat{z} \in \widehat{F}^{(n)}$ be a horocycle and define $\hat{c}_{\hat{z}} := \sum_{i = 0}^\infty \widehat{H} \widehat{\varphi}^i(\hat{z})$. Then $\hat{c}_{\hat{z}}$ is a horochain and $\partial \hat{c}_{\hat{z}} = \hat{z}$.
\end{lem}

\begin{proof}
    By \cref{lem:ExtendHom}\ref{item:valIneq} there are constants $\alpha$ and $\beta$ such that
    \[
        v(\widehat{\varphi}(\hat{c})) \geqslant v(\hat{c}) + \alpha \ \ \text{and} \ \ v(\widehat{H}(\hat{c})) \geqslant v(\hat{c}) + \beta
    \]
    for every horochain $\hat{c} \in \widehat{F}^{(n)}$. Moreover, $\alpha > 0$ since $v(\varphi(f)) > v(f)$ for every $f \in F^{(n)}$. To see that $\hat{c}$ is a horochain, by induction we have $v(\widehat{H}\widehat{\varphi}^i(\hat{z})) \geqslant v(\hat{z}) + i\alpha + \beta$, so for all $t \in \R$ there are only finitely many integers $i \geqslant 0$ such that $v(\widehat{H} \widehat{\varphi}^i(\hat{z})) \leqslant t$. Since $\supp(\hat{c}_{\hat{z}}) \subseteq \bigcup_{i=0}^\infty \supp(\widehat{H}\widehat{\varphi}^i(\hat{z}))$ and each $\widehat{H} \widehat{\varphi}^i(\hat{z})$ is a horochain, it follows that there are only finitely many $gx \in \supp \hat{c}$ such that $v(gx) \leqslant t$, so $\hat{c}$ is a horochain.

    Finally, we have
    \begin{equation*}
        \partial\hat{c} = \sum_{i=0}^\infty \partial \widehat{H} \widehat{\varphi}^i(\hat{z}) = \sum_{i=0}^\infty (\id_{\widehat{F}^{(n)}} - \widehat{\varphi} - \widehat{H}\partial)\widehat{\varphi}^i(\hat{z}) = \sum_{i=0}^\infty (\widehat{\varphi}^i - \widehat{\varphi}^{i+1})(\hat{z}) = \hat{z}. \qedhere
    \end{equation*}
\end{proof}

\begin{proof}[Proof of \ref{item:LiftId} $\Rightarrow$ \ref{item:HoroAcyc}]
    By \cite[Lemma I.7.4]{BrownGroupCohomology}, there is a chain homotopy $H \colon F \rightarrow F$ such that $\partial H + H \partial = \id_F - \varphi$. If $\hat{z} \in \widehat{F}^{(n)}$ is a horocycle, then $\partial \hat{c}_{\hat{z}} = \hat z$ by \cref{lem:Useful}. 
\end{proof}

\begin{proof}[Proof of \ref{item:HoroAcyc} $\Rightarrow$ \ref{item:EssAcyc}] 
    We will prove that $F_\bullet^v \rightarrow M \rightarrow 0$ is essentially acyclic in dimension $k$ for all $k < n$ by induction on $k$. For the base case, we show that $F_\bullet^v \rightarrow M \rightarrow 0$ is exact at $M$, which implies essential acyclicity in dimension $-1$. Let $m \in M$. By exactness of $F_\bullet \rightarrow M \rightarrow 0$, there is a chain $c \in F_0$ such that $\partial c = m$. By horo-acyclicity in dimension $0$, there is some horochain $\hat{c} \in \widehat{F}_1$ such that $\partial \hat{c} = c$. There are $c_{-} \in F_1$ and $\hat{c}_+ \in \widehat{F}_1$ such that  $\hat{c} = c_{-} + \hat{c}_{+}$, where $v(c_{-}) < 0$ and $v(\hat{c}_{+}) \geqslant 0$. Then $\partial(c - \partial c_-) = m$ and 
    \[
        v(c - \partial c_-) = v(c - \partial(\hat{c} - \hat{c}_+)) = v(\partial \hat{c}_+) \geqslant v(\hat{c}_+) \geqslant 0.
    \]
    This shows that $c - \partial c_0 \in F_0^v$, which proves that $F_\bullet^v \rightarrow M \rightarrow 0$ is exact at $M$.

    Let $k > -1$ and suppose that $F_\bullet^v \rightarrow M \rightarrow 0$ is essentially acyclic in dimensions $< k$. By \ref{item:EssAcyc} $\Rightarrow$ \ref{item:LiftId} applied at $k-1$, there is a chain map $\varphi \colon F \rightarrow F$ lifting $\id_M$ such that $v(\varphi(c)) > v(c)$ for all $c \in F^{(k)}$. Since $\id_F$ and $\varphi$ both lift $\id_M$ and $F_\bullet \rightarrow M \rightarrow 0$ is acyclic, there is a chain homotopy $H \colon F \rightarrow F$ such that $\partial H + H \partial = \id_{F} - \varphi$ (see \cite[Lemma I.7.4]{BrownGroupCohomology}). As in the proof of \cref{lem:Useful}, there are constants $\alpha > 0$ and $\beta < 0$ such that 
    \[
        v(\varphi(c)) \geqslant v(c) + \alpha \ \ \text{and} \ \ v(H(c)) \geqslant v(c) + \beta
    \]
    for every $c \in F^{(k)}$.

    Let $z \in F^v_{k}$ be a cycle. Since $F_\bullet \rightarrow M \rightarrow 0$ is acyclic, there is some $d \in F_{k+1}$ such that $\partial d = z$. Consider the horocycle $\hat{z} := d - \hat{d}_z$, where $\hat{d}_z = \sum_{i = 0}^\infty H\varphi^i(z)$ is defined as in \cref{lem:Useful}. Note that
    \[
        v(H\varphi^i(z)) \geqslant v(z) + i\alpha + \beta \geqslant \beta
    \]
    for every $i \geqslant 0$, and therefore that $v(\hat{d}_z) \geqslant \beta$. By horo-acyclicity in dimension $k+1$, there is a $(k+2)$-horochain $\hat{d}$ such that $\partial \hat{d} = \hat{z}$. As in the base case, there are $d_- \in F_{k+2}$ and $\hat{d}_+ \in \widehat{F}_{k+2}$ such that $\hat{d} = d_- + \hat{d}_+$, where $v(d_-) < 0$ and $v(\hat{d}_+) \geqslant 0$. Then $\partial(d - \partial d_-) = \partial d = z$, and
    \begin{align*}
        v(d - \partial d_-) &= v(\hat{d}_z + \hat{z} - \partial(\hat{d} - \hat{d}_+)) \\
            &= v(\hat{d}_z + \partial \hat{d}_+) \\
            &\geqslant \min\{ v(\hat{d}_z), v(\partial \hat{d}_+) \} \\
            &\geqslant \beta
    \end{align*}
    since $v(\hat{d}_z) \geqslant \beta$ and $v(\partial d_\infty) \geqslant v(d_\infty) \geqslant 0 > \beta$. Letting $D = -\beta$ in the definition of essential acyclicity, we see that $F_\bullet^v \rightarrow M \rightarrow 0$ is essentially acyclic in dimension $k$. \qedhere
\end{proof}

\begin{proof}[Proof of \ref{item:TorCond} $\Rightarrow$ \ref{item:HoroAcyc}] 
    Suppose that $\Tor_i^{RG}(\widehat{RG}^\chi,M) = 0$ for $0 \leqslant i \leqslant n$. Consider the chain map
    \[
        \psi \colon \widehat{RG}^\chi \otimes_{RG} F \rightarrow \widehat{F}, \ \ \alpha \otimes c \mapsto \alpha c
    \]
    of left $RG$-modules. It is clear that $\psi$ is injective. We claim that $\psi$ induces an isomorphism $\widehat{RG}^\chi \otimes_{RG} F^{(n)} \rightarrow \widehat{F}^{(n)}$. To see this, simply note that for an arbitrary horochain
    \[
        \hat{c} = \sum_{g \in G, x \in X^{(n)}} r_{g,x} gx
    \]
    in $\widehat{F}^{(n)}$, we have
    \[
        \sum_{x \in X^{(n)}}  \left( \sum_{g\in G} r_{g,x} g \right) \otimes x \xmapsto{\varphi} \hat{c}.
    \]
    The horochain condition implies that the sums $\sum_{g\in G} r_{g,x} g$ are elements of $\widehat{RG}^\chi$. Thus, $\psi$ is surjective on the $n$-skeleta and is therefore an isomorphism. Note that this only works because $X^{(n)}$ is finite; in general, we cannot expect $\psi$ to be surjective since the support of a horochain might intersect infinitely many of the modules $F_n$. Since $\Tor_i^{RG}(\widehat{RG}^\chi \otimes_{RG} F, M) = 0$ for all $0 \leqslant i \leqslant n$, we conclude that $\Tor_i(\widehat{F}, M) = 0$ for $0 \leqslant i \leqslant n$ as well.
\end{proof}

\begin{proof}[Proof of \ref{item:HoroAcyc}
    $\Rightarrow$ \ref{item:TorCond}] The map $\psi$ defined above is an isomorphism of the $n$-skeleta, so we immediately have that $\Tor_i^{RG}(\widehat{RG}^\chi,M) = 0$ for $0 \leqslant i \leqslant n-1$. Since $\psi$ is not necessarily surjective as a map of the $(n+1)$-skeleta, we must work harder to show that $\Tor_n^{RG}(\widehat{RG}^\chi,M) = 0$. Let $z \in \widehat{RG}^\chi \otimes_{RG} F_n$ be an $n$-cycle, and let $\hat{z} = \psi(z)$. Since we are assuming that \ref{item:HoroAcyc} holds, we may also assume that \ref{item:LiftId} holds and use the horochain $\hat{c}_{\hat{z}}$ from \cref{lem:Useful}. Since $\hat{c}_{\hat{z}} \in \widehat{H}(\widehat{F}_n)$, we have that $\hat{c}_{\hat{z}}$ is in the $\widehat{RG}^\chi$-submodule of $\widehat{F}_{n+1}$ generated by $\widehat{H}(X_n)$, and thus $\hat{c}_{\hat{z}} \in \im \psi$ since this is a finite set. Let $c \in \widehat{RG}^\chi \otimes_{RG} F_{n+1}$ such that $\psi(c) = \hat{c}_{\hat{z}}$. Then $\psi \partial (c) = \partial \psi(c) = \partial \hat{c}_{\hat{z}} = \hat{z}$. But $\psi$ is injective, so $\partial c = z$, proving that $\Tor_n^{RG}(\widehat{RG}^\chi,M) = 0$.
\end{proof}

We pause again before proving the equivalence of \ref{item:SigmaInv} and \ref{item:EssAcyc} to prove another lemma. 

\begin{lem}\label{lem:flat}
    Free $RG$-modules are flat over $RG_\chi$.
\end{lem}

\begin{proof}
    It suffices to prove that $RG$ is flat as an $RG_\chi$-module, since the direct sum of flat modules is flat. To this end, let $\iota \colon M \hookrightarrow N$ be an injection of right $RG_\chi$-modules; our goal is to show that $\iota \otimes \id \colon M \otimes_{RG_\chi} RG \rightarrow N \otimes_{RG_\chi} RG$ is injective. Let $g \in G$ be such that $\chi(g) < 0$ and consider the left $RG_\chi$-module $RG_\chi g^k = \{ \alpha g^k : \alpha \in RG_\chi, k \in \Z \}$. The modules $RG_\chi g^k$ form a directed system with respect to the inclusion maps $RG_\chi g^k \hookrightarrow RG_\chi g^l$ for $k \leqslant l$ and the direct limit is $\varinjlim RG_\chi g^k \cong RG$.

    There are left $RG_\chi$-module isomorphisms $RG_\chi g^k \rightarrow RG_\chi$ given by right multiplication by $g^{-k}$. Then $RG_\chi g^k$ is flat over $RG_\chi$, so $M \otimes_{RG_\chi} RG_\chi g^k \rightarrow N \otimes_{RG_\chi} RG_\chi g^k$ is injective for all $k \in \Z$. By exactness of the direct limit,
    \[
        \varinjlim(M \otimes_{RG_\chi} RG_\chi g^k) \rightarrow \varinjlim(N \otimes_{RG_\chi} RG_\chi g^k)
    \]
    is injective. Since the direct limit commutes with the tensor product, the previous line implies $\iota \otimes \id_M$ is injective. \qedhere
\end{proof}

We now return to the proof of \cref{thm:Main}.

\begin{proof}[Proof of \ref{item:SigmaInv} $\Leftrightarrow$ \ref{item:EssAcyc}] 
    Let $g \in G$ be such that $\chi(g) < 0$ and let $E_k$ be the left $RG_\chi$-module $g^k F^v$. We denote the chain complexes $F^v_\bullet \rightarrow M \rightarrow 0$ and $(E_k)_\bullet \rightarrow M \rightarrow 0$ by $\widetilde{F}^v$ and $\widetilde{E}_k$, respectively.

    Essential acyclicity in dimension $j$ is equivalent to the existence of an integer $D \geqslant 0$ such that the inclusion-induced homomorphism $H_j(\widetilde{E}_k) \rightarrow H_j(\widetilde{E}_{k+D})$ is the zero map for all $k \in \N$. This in turn is equivalent to $\varinjlim \prod_{I} H_j(\widetilde{E}_k) = 0$ for any index set $I$. Here, for fixed $I$ and $j$, the powers $\prod_{I} H_j(\widetilde{E}_k)$ form a directed system with respect to the inclusion-induced maps $\prod_I H_j(\widetilde{E}_k) \rightarrow \prod_I H_j(\widetilde{E}_l)$ for $k \leqslant l$. Indeed, if $D \geqslant 0$ is such that $H_j(\widetilde{E}_k) \rightarrow H_j(\widetilde{E}_{k+D})$ is the zero map, it is clear that the direct limit will be zero. Conversely, let $I = Z_j(\widetilde{E}_0) = Z_j(\widetilde{F}^v)$ be the set of $j$-cycles of $\widetilde{F}^v$ and consider the element $([x])_{x \in I} \in \prod_I H_j(\widetilde{E}_0)$. Since the direct limit is zero, there is some $D \geqslant 0$ such that $([x])_{x \in I} = 0$ in $\prod_I H_j(\widetilde{E}_D)$, which means that $\widetilde{F}^v$ is essentially acyclic in dimension $j$.

    There is a short exact sequence of chain complexes $0 \rightarrow M \rightarrow \widetilde{E}_k \rightarrow E_k \rightarrow 0$, where, by abuse of notation, $M$ is a chain complex concentrated in dimension $-1$ and $E_k$ is the chain complex $(E_k)_\bullet \rightarrow 0$ with $(E_k)_0$ in dimension $0$. The long exact sequence in homology associated to the short exact sequence gives $H_j(\widetilde{E}_k) \cong H_j(E_k)$ for $j \geqslant 1$. The interesting part of the long exact sequence is
    \[
        0 \rightarrow H_0(\widetilde{E}_k) \rightarrow H_0(E_k) \xrightarrow{\delta} M \rightarrow H_{-1}(\widetilde{E}_k) \rightarrow 0,
    \]
    where $\delta$ is the connecting homomorphism. By exactness of the direct power and direct limit functors, the sequence
    \[
        0 \rightarrow \varinjlim \prod_I H_0(\widetilde{E}_k) \rightarrow \varinjlim \prod_I H_0(E_k) \xrightarrow{\prod_I \delta} \prod_I M \rightarrow \varinjlim \prod_I H_{-1}(\widetilde{E}_k) \rightarrow 0
    \]
    is exact. Then $\widetilde{F}^v$ is essentially acyclic in dimension $0$ if and only if $\delta$ induces an injection $\varinjlim \prod_I H_0(E_k) \rightarrow \prod_I M$ for every $I$. Moreover, $\widetilde{F}^v$ is essentially acyclic in dimension $-1$ if and only if $\delta$ induces a surjection $\varinjlim \prod_I H_0(E_k) \rightarrow \prod_I M$ for every $I$.

    By \cref{lem:flat},  $F_\bullet \rightarrow M \rightarrow 0$ is a flat resolution of $M$ by left $RG_\chi$-modules, so 
    \[
        \Tor^{RG_\chi}_j \left(\prod_I RG_\chi, M \right) = H_j\left( \left(\prod_I RG_\chi \right) \otimes_{RG_\chi} F \right),
    \]
    and therefore
    \[
        \Tor^{RG_\chi}_j \left(\prod_I RG_\chi, M \right) = \varinjlim H_j \left( \left(\prod_I RG_\chi \right) \otimes_{RG_\chi} E_k \right),
    \]
    as $F = \varinjlim E_k$ and direct limits commute with tensor products and homology. Since $(E_k)_j$ is a finitely generated free $RG_\chi$-module for $j \leqslant n$, we have $\left(\prod_I RG_\chi \right) \otimes_{RG_\chi} (E_k)_j \cong \prod_I (E_k)_j$. Hence, $\Tor_j^{RG_\chi}(\prod_I RG_\chi,M) = \varinjlim H_j(\prod_I E_k)$ for $j < n$.

    To summarise the work done above, we have $\widetilde{F}^v$ is essentially acyclic in dimensions $-1 \leqslant j < n$ if and only if
    \begin{enumerate}[label=(\alph*)]
        \item $(\prod_I RG_\chi) \otimes_{RG_\chi} M \rightarrow \prod_I M$ is surjective if $n = 0$ and
        \item $(\prod_I RG_\chi) \otimes_{RG_\chi} M \rightarrow \prod_I M$ is an isomorphism and 
        \[
            \Tor_j^{RG_\chi}(\prod_I RG_\chi,M)
        \]
        vanishes for $1 \leqslant j < n$ otherwise.
    \end{enumerate}
    Here we have used the general fact that $\Tor_0^R(A,B) \cong A \otimes_R B$. Together with Lemma 1.1 and Proposition 1.2 of \cite{BieriEckmannFinProps}, (a) and (b) are equivalent to $M$ being of type $\mathtt{FP}_n(RG_\chi)$. Thus, we conclude that $[\chi] \in \Sigma^m_R(G;M)$ if and only $\widetilde{F}^v$ is essentially acyclic in dimensions $j = -1, 0, 1, \dots, n-1$. \qedhere
\end{proof}

\section{Agrarian homology and main result} \label{sec:homology}

\begin{defn}[agrarian groups and $\mathcal{D}$-homology]
    Let $R$ be a ring. A group $G$ is \textit{agrarian over $R$} if there is a skew-field $\mathcal{D}$ and an injective ring homomorphism $R G \hookrightarrow \mathcal{D}$. In this case, we will say that $G$ is $\mathcal{D}$-agrarian over $R$ if we wish to specify the skew-field.

    If $G$ is $\mathcal{D}$-agrarian over $R$, we define its $p$-dimensional \textit{$\mathcal{D}$-homology} to be
    \[
        H_p^{\mathcal{D}} (G) := \Tor_p^{RG} (\mathcal{D}, R),
    \]
    where $R$ is the trivial $RG$-module and $\mathcal{D}$ is viewed as a $\mathcal{D}$-$RG$-bimodule via the embedding $RG \hookrightarrow \mathcal{D}$. The $p$th $\mathcal{D}$-Betti number of $G$ is then
    \[
        b_p^{\mathcal{D}}(G) := \dim_\mathcal{D} H_p^{\mathcal{D}} (G).
    \]
    Note that $b_p^{\mathcal{D}}(G)$ is well-defined and integral or infinite, since a module over a skew-field has a well-defined dimension.
\end{defn}

\begin{rem}
    The term ``agrarian" was introduced by Kielak in \cite{KielakBNSviaNewton} in the case $R = \mathbb{Z}$. Using strong Hughes-freeness, i.e.~condition \ref{item:2prime} after \cref{def:HfreeDiv}, it follows that if $G$ is a locally indicable group and $\mathbb{F}$ is a skew-field such that $\mathcal{D}_{\mathbb{F}G}$ exists, then $G$ is $\mathcal{D}_{\mathbb{F}G}$-agrarian over $\mathbb{F}$. If $R$ is a skew-field or an integral domain, then there are no known examples of torsion-free groups that are not agrarian over $R$. For the remainder of the section, we will be interested in the $\mathcal{D}_{\mathbb{F}G}$-homology of $G$.
\end{rem}

In what follows we will need that $\mathcal{D}_{\mathbb{F}G}$-Betti numbers have good scaling properties when passing to finite index subgroups. This is analogous to the fact that $\ell^2$-Betti numbers also scale under passage to finite index subgroups.

\begin{lem}\label{lem:JBscales}
    Let $H$ be a finite index subgroup of $G$ and let $\mathbb{F}$ be a skew-field such that $\mathcal{D}_{\mathbb{F}G}$ exists. Then 
    \[
        b_p^{\mathcal{D}_{\mathbb{F}G}}(G) = \frac{b_p^{\mathcal{D}_{\mathbb{F}H}}(H)}{[G:H]}.
    \]
\end{lem}

\begin{proof} 
    It suffices to prove the claim when $H$ is normal in $G$. To see this, if $H \leqslant G$ is any subgroup of finite index, then there is normal subgroup $N \triangleleft G$ of finite index such that $N \leqslant H \leqslant G$ (we can take $N$ to be the \textit{normal core} of $H$, that is, the intersection of all the conjugates of $H$). Then
    \[
    b_p^{\mathcal{D}_{\mathbb{F}G}}(G) = \frac{b_p^{\mathcal{D}_{\mathbb{F}N}}(N)}{[G:N]} = \frac{[H:N] b_p^{\mathcal{D}_{\mathbb{F}H}}(H)}{[G:N]} = \frac{b_p^{\mathcal{D}_{\mathbb{F}H}}(H)}{[G:H]},
    \]
    by the claim for normal subgroups.

    Assume that $H$ is a finite index normal subgroup of $G$ and let $\{t_1, \dots, t_n \}$ be a transversal for $H$ in $G$. By \cref{prop:twistedNormalSkew}\ref{item:twistedFiniteIndex}, we have a $\mathcal{D}_{\mathbb{F}H}$-$\mathbb{F}G$-bimodule isomorphism $\mathcal{D}_{\mathbb{F}G} \cong  \mathcal{D}_{\mathbb{F}H} \otimes_{\mathbb{F}H} \mathbb{F}G$. Take a free resolution $F_\bullet \rightarrow \mathbb{F} \rightarrow 0$ of the trivial left $\mathbb{F}G$-module $\mathbb{F}$; there are chain isomorphisms
    \[
        \mathcal{D}_{\mathbb{F}G} \otimes_{\mathbb{F}G} F_\bullet 
        \cong (\mathcal{D}_{\mathbb{F}H} \otimes_{\mathbb{F}H} \mathbb{F}G) \otimes_{\mathbb{F}G} F_\bullet  
        \cong \mathcal{D}_{\mathbb{F}H} \otimes_{\mathbb{F}H} (\mathbb{F}G \otimes_{\mathbb{F}G} F_\bullet) 
        \cong \mathcal{D}_{\mathbb{F}H} \otimes_{\mathbb{F}H} F_\bullet
    \]
    of left $\mathcal{D}_{\mathbb{F}H}$-modules, so $H_p^{\mathcal{D}_{\mathbb{F}G}}(G) \cong H_p^{\mathcal{D}_{\mathbb{F}H}}(H)$ as $\mathcal{D}_{\mathbb{F}H}$-modules. Therefore,
    \begin{align*}
        b_p^{\mathcal{D}_{\mathbb{F}H}}(H) &= \dim_{\mathcal{D}_{\mathbb{F}H}} H_p^{\mathcal{D}_{\mathbb{F}H}}(H) \\
            &= \dim_{\mathcal{D}_{\mathbb{F}H}} H_p^{\mathcal{D}_{\mathbb{F}G}}(G) \\
            &= [G : H] \cdot \dim_{\mathcal{D}_{\mathbb{F}G}} H_p^{\mathcal{D}_{\mathbb{F}G}}(G) \\ 
            &= [G : H] \cdot b_p^{\mathcal{D}_{\mathbb{F}G}}(G). \qedhere
    \end{align*}
\end{proof}

In view of \cref{lem:JBscales}, if $G$ is a group with a finite index subgroup $H$ such that $\mathcal{D}_{\mathbb{F}H}$ exists, we can define $b_p^{\mathcal{D}_{\mathbb{F}G}}(G) = b_p^{\mathcal{D}_{\mathbb{F}H}}(H)/[G:H]$. This is an abuse of notation  since $\mathcal{D}_{\mathbb{F}G}$ might not exist.

The following is an analogue of a theorem of Lück which holds for $\ell^2$-Betti numbers \cite[Theorem 7.2]{Luck02}.

\begin{thm}\label{thm:JBSES}
Let $1 \rightarrow K \rightarrow G \rightarrow \Z \rightarrow 1$ be a short exact sequence of groups and suppose that $\mathcal{D}_{\mathbb{F}G}$ exists for some skew-field $\mathbb{F}$. If $b_p^{\mathcal{D}_{\mathbb{F}K}}(K) < \infty$ for some $p \geqslant 0$, then $b_p^{\mathcal{D}_{\mathbb{F}G}}(G) = 0$. 
\end{thm}
\begin{proof}
Let $F_\bullet \rightarrow \mathbb{F} \rightarrow 0$ be a free resolution of $\mathbb{F}$ by left $\mathbb{F} G$-modules. Note that the modules $F_j$ are also free left $\mathbb{F}K$-modules and there are chain maps $\iota_j \colon \mathcal{D}_{\mathbb{F}K} \otimes_{\mathbb{F}K} F_j \rightarrow \mathcal{D}_{\mathbb{F}G} \otimes_{\mathbb{F}G} F_j$, induced by the inclusion $\mathcal{D}_{\mathbb{F}K} \hookrightarrow \mathcal{D}_{\mathbb{F}G}$. We claim that the maps $\iota_j$ are injective. To see this, it is enough to consider the case where $F_j = \mathbb{F}G$. Choose $t \in G$ such that $\{ t^n : n \in \Z \}$ is a transversal for $K \leqslant G$. By \cref{prop:twistedNormalSkew}, there is an embedding $\bigoplus_{n \in \Z} \mathcal{D}_{\mathbb{F}K} \cdot \varphi(t^n) \hookrightarrow \mathcal{D}_{\mathbb{F}G}$. Since $\mathbb{F}G$ is free over $\mathbb{F}K$, there is also an isomorphism $\mathcal{D}_{\mathbb{F}K} \otimes_{\mathbb{F}K} \mathbb{F}G \cong \bigoplus_{n \in \Z} \mathcal{D}_{\mathbb{F}K} \cdot \varphi(t^n)$ determined by $\alpha \otimes t^n \mapsto \alpha \varphi(t^n)$. Then the diagram
\[
\begin{tikzcd}
\mathcal{D}_{\mathbb{F}K} \otimes_{\mathbb{F}K} \mathbb{F}G \arrow[d, "\iota_j"'] \arrow[r, "\cong", no head] & {\bigoplus_{n \in \mathbb{Z}} \mathcal{D}_{\mathbb{F}K} \cdot \varphi(t^n)}  \arrow[d, hook] \\
\mathcal{D}_{\mathbb{F}G} \otimes_{\mathbb{F}G} \mathbb{F}G \arrow[r, "\cong", no head]                       &  \mathcal{D}_{\mathbb{F}G}
\end{tikzcd}
\]
of left $\mathcal{D}_{\mathbb{F}K}$-modules commutes, proving that $\iota_j$ is an injection. From now on, we will treat the maps $\iota_j$ as inclusions.

Consider the following portions of the chain complexes computing $b_p^{\mathcal{D}_{\mathbb{F}G}}(G)$ and $b_p^{\mathcal{D}_{\mathbb{F}K}}(K)$
\[
\begin{tikzcd}[column sep = small]
\cdots \arrow[r] & \mathcal{D}_{\mathbb{F}K} \otimes_{\mathbb{F}K} F_{p+1} \arrow[d, hook, "\iota_{p+1}"] \arrow[r] & \mathcal{D}_{\mathbb{F}K} \otimes_{\mathbb{F}K} F_p \arrow[d, hook, "\iota_p"] \arrow[r] & \mathcal{D}_{\mathbb{F}K} \otimes_{\mathbb{F}K} F_{p-1} \arrow[d, hook, "\iota_{p-1}"] \arrow[r] & \cdots \\
\cdots \arrow[r] & \mathcal{D}_{\mathbb{F}G} \otimes_{\mathbb{F}G} F_{p+1} \arrow[r]                               & \mathcal{D}_{\mathbb{F}G} \otimes_{\mathbb{F}G} F_p \arrow[r]                           & \mathcal{D}_{\mathbb{F}G} \otimes_{\mathbb{F}G} F_{p-1} \arrow[r]                                 & \cdots \nospacepunct{.}
\end{tikzcd}
\]
Let $x$ be a cycle in $\mathcal{D}_{\mathbb{F}G} \otimes_{\mathbb{F}G} F_p$. By \cite[Proposition 2.2(2)]{JaikinZapirain2020THEUO}, $\mathcal{D}_{\mathbb{F}G} \cong \mathrm{Ore}(\mathcal{D}_{\mathbb{F}K} * \Z)$, where we are making the identifications $\mathcal{D}_{\mathbb{F}K} \otimes_{\mathbb{F}K} \mathbb{F}G \cong \bigoplus_{n \in \Z} \mathcal{D}_{\mathbb{F}K} \cdot \varphi(t^n) \cong \mathcal{D}_{\mathbb{F}K} * \Z$. Hence, there is some nonzero $a \in \mathcal{D}_{\mathbb{F}K} * \Z$ such that $ax \in \mathcal{D}_{\mathbb{F}K} \otimes_{\mathbb{F}K} F_p$. Since $(\mathcal{D}_{\mathbb{F}K} * \Z) \cdot ax \subseteq Z_p(\mathcal{D}_{\mathbb{F}G} \otimes_{\mathbb{F}G} F_\bullet)$ and $\iota_{p-1}$ is injective, $(\mathcal{D}_{\mathbb{F}K} * \Z) \cdot ax$ is an infinite-dimensional $\mathcal{D}_{\mathbb{F}K}$-subspace of $Z_p(\mathcal{D}_{\mathbb{F}K} \otimes_{\mathbb{F}K} F_\bullet)$. Since $b_p^{\mathcal{D}_{\mathbb{F}K}}(K) < \infty$, there is a nonzero $b \in \mathcal{D}_{\mathbb{F}K} * \Z$ such that $bax = \partial y$ for some $y \in \mathcal{D}_{\mathbb{F}K} \otimes_{\mathbb{F}K} F_{p+1}$. But then $x = \partial((ba)\inv y)$, so we conclude that $H_p(\mathcal{D}_{\mathbb{F}G} \otimes_{\mathbb{F}G} F_\bullet) = 0$. \qedhere
\end{proof}

For the proof of the main theorem, we will need the following version of a theorem of Bieri and Renz. The details of the proof are given in \cite[Theorem 5.1]{BieriRenzValutations} in the case $R = \Z$, though the proof goes through in exactly the same way after replacing the ring $\Z$ by an arbitrary ring $R$.

\begin{thm}[Bieri-Renz]\label{thm:genBR}
Let $G$ be a finitely generated group and let $N \triangleleft G$ be a normal subgroup containing the commutator subgroup $[G,G]$. Let $R$ be a unital ring and let $M$ be an $RG$-module of type $\mathtt{FP}_n(RG)$. Then $M \in \mathtt{FP}_n(RN)$ if and only if $\Sigma_R^n(G;M) \supseteq S(G,N) := \{ [\chi] \in S(G) : \chi(N) = 0 \}$.
\end{thm}

We will also need the following result due to Kielak and Jaikin-Zapirain. Kielak first proved the result in \cite[Theorem 5.2]{KielakRFRS} by giving an explicit construction of the Linnell skew-field $\mathcal D(G)$ when $G$ is RFRS. In the appendix to \cite{JaikinZapirain2020THEUO}, Jaikin-Zapirain showed that when $G$ is RFRS and $\mathbb F$ is any skew-field, then $\mathcal D_{\mathbb F G}$ exists and admits a completely analogous construction to $\mathcal D(G)$ (in fact $\mathcal D(G) = \mathcal D_{\Q G}$ for a RFRS group $G$). As a consequence, Kielak's proof of \cref{thm:KJZ} still holds after making the replacements $\Q \rightsquigarrow \mathbb F$ and $\mathcal D(G) \rightsquigarrow \mathcal D_{\mathbb FG}$.

\begin{thm}[Kielak, Jaikin-Zapirain] \label{thm:KJZ}
Let $\mathbb F$ be a skew-field, $G$ a finitely generated RFRS group, and $n \in \N$. Let $F_\bullet$ be a chain complex of free $\mathbb F G$-modules with $F_p$ is finitely generated and $H_p(\mathcal D_{\mathbb F G} \otimes_{\mathbb F G} F_\bullet) = 0$ for all $p \leqslant N$. Then, there exist a finite index subgroup $H \leqslant G$ and an open subset $U \subseteq S(H)$ such that
\begin{enumerate}
    \item the closure of $U$ contains $S(G)$;
    \item $U$ is invariant under the antipodal map;
    \item $H_p(\widehat{\mathbb F H}^\chi \otimes_{\mathbb F H} F_\bullet) = 0$ for every $p \leqslant n$ and every $[\chi] \in U$.
\end{enumerate}
\end{thm}

We are now ready to prove the main theorem.

\begin{thm}\label{thm:agrarianMain}
Let $\mathbb{F}$ be a skew-field and let $G$ be a virtually RFRS group of type $\mathtt{FP}_n(\mathbb{F})$. Then there is a finite index subgroup $H \leqslant G$ admitting a homomorphism onto $\Z$ with kernel of type $\mathtt{FP}_n(\mathbb{F})$ if and only if $b_p^{\mathcal{D}_{\mathbb{F}G}}(G) = 0$ for $p = 0, \dots, n$.
\end{thm}
\begin{proof}
($\Rightarrow$) Let $\varphi \colon H \rightarrow \Z$ be an epimorphism with kernel $K \in \mathtt{FP}_n(\mathbb{F})$. Then there is a free resolution $F_\bullet \rightarrow \mathbb{F} \rightarrow 0$ of the trivial $\mathbb{F}K$-module $\mathbb{F}$ with finitely generated $n$-skeleton. Therefore, 
\[
    b_p^{\mathcal{D}_{\mathbb{F}K}}(K) = \dim_{\mathcal{D}_{\mathbb{F}K}} H_p (\mathcal{D}_{\mathbb{F}K} \otimes_{\mathbb{F}K} F_\bullet) < \infty
\]
for $p \leqslant n$ and we have a short exact sequence $1 \rightarrow K \rightarrow H \rightarrow \Z \rightarrow 1$, so $b_p^{\mathcal{D}_{\mathbb{F}H}}(H) = 0$ for $p \leqslant n$ by \cref{thm:JBSES}. Then $b_p^{\mathcal{D}_{\mathbb{F}G}}(G) = 0$ for $p \leqslant n$ by \cref{lem:JBscales}.

\smallskip

($\Leftarrow$) The properties of being of type $\mathtt{FP}_n(\mathbb{F})$ and of having vanishing $p$th $\mathcal{D}_{\mathbb{F}G}$-Betti number pass to finite index subgroups. Moreover, the property of being virtually fibred passes to finite index overgroups (the same is of course true of any virtual property). Hence, we may assume that $G$ is RFRS and of type $\mathtt{FP}_n(\mathbb{F})$. Then $H^1(G; \R) \neq 0$ by \cref{prop:RFRStoQ}.

Since $G \in \mathtt{FP}_n(\mathbb{F})$, there is a free resolution $F_\bullet \rightarrow \mathbb{F} \rightarrow 0$ of the trivial $\mathbb{F}G$-module $\mathbb{F}$ with $F_p$ finitely generated for $p \leqslant n$. By assumption, $H_p(\mathcal{D}_{\mathbb{F}G} \otimes_{\mathbb{F}G} F_\bullet) = 0$ for $p \leqslant n$. By \cref{thm:KJZ}, there is a finite index subgroup $H \leqslant G$ and an open subset $U \subseteq H^1(H;\R)$ such that the closure of $U$ contains $H^1(G;\R)$, is invariant under nonzero scalar multiplication, and $H_p(\widehat{\mathbb{F}H}^\varphi \otimes_{\mathbb{F}H} F_\bullet) = 0$ for $p \leqslant n$ and all $\phi \in U$. Since $U$ is nonempty, we can find a surjective character $\varphi \colon H \rightarrow \Z$ in $U$. To see this, let $\varphi \in H^1(G;\R)$ be a nontrivial character. Since $U$ is open, $\varphi$ can be perturbed so that its image is in $\Q$. Finally, since $H$ is finitely generated, we can rescale the character so that it maps $H$ onto $\Z$.  Then $[\pm \varphi] \in \Sigma_\mathbb{F}^n(H;\mathbb{F})$ by \cref{thm:Main}. It is not hard to show that $[\chi] = [\pm \varphi]$ for any character $\chi \colon H \rightarrow \R$ with $\ker \varphi \subseteq \ker \chi$. Thus, $\ker \varphi \in \mathtt{FP}_n(\mathbb{F})$ by \cref{thm:genBR}. \qedhere

\end{proof}

\begin{cor}\label{cor:typeFP}
    Let $G$ be virtually RFRS and of type $\FP(\mathbb{F})$. Then $G$ virtually algebraically fibres with kernel of type $\mathtt{FP}(\mathbb F)$ if and only if $G$ is $\DF{G}$-acyclic.
\end{cor}

\begin{proof}
    One direction is clear. If $G$ is $\DF{G}$-acyclic, then $G$ virtually algebraically fibres with kernel $K$ of type $\FP_n(\mathbb F)$ for $n > \cd_\mathbb F(G)$. But then $n > \cd_\mathbb F (K)$, so $K$ is of type $\FP(\mathbb F)$. \qedhere
\end{proof}

We now apply \cref{thm:agrarianMain} to the case $\mathbb{F} = \mathbb{Q}$.

\begin{defn}[$\ell^2$-Betti numbers] \label{def:l2b}
Let $G$ be a torsion-free group satisfying the Atiyah conjecture and let $\mathcal{D}(G)$ be the Linnell skew-field of $G$ (see the following remark). Define
\[
    b_p^{(2)}(G) = \dim_{\mathcal{D}(G)} \Tor_p^{\Q G} (\mathcal{D}(G), \Q)
\]
to be the \textit{$p$th $\ell^2$-Betti number of $G$}.

If a group $G$ has a torsion-free subgroup $H$ of finite index satisfying the Atiyah conjecture, we extend the definition of $\ell^2$-Betti numbers by declaring that
\[
    b_n^{(2)}(G) = \frac{b_n^{(2)}(H)}{[G:H]}.
\]
\end{defn}

\begin{rem}
This definition of $\ell^2$-Betti numbers for torsion-free groups satisfying the Atiyah conjecture agrees with the usual definition by \cite[Lemma 10.28(3)]{Luck02}. Moreover, $\ell^2$-Betti numbers for virtually torsion-free groups are well-defined and coincide with the usual definition by \cite[Theorem 6.54(6)]{Luck02}. We will not give the definition of the Linnell ring $\mathcal{D}(G)$ since that would take us too far afield. We take the Atiyah conjecture to be the statement that $\mathcal{D}(G)$ is a skew-field, which allows us to make \cref{def:l2b}. Linnell showed that this formulation implies the strong Atiyah conjecture over $\Q$ for torsion-free groups \cite{LinnellDivRings93}. The details of the reverse implication can be found in L\"uck's book \cite[Section 10]{Luck02}.
\end{rem}

In the case where $G$ is finitely generated and RFRS, $\mathcal{D}_{\Q G}$ exists and is isomorphic to the Linnell skew-field $\mathcal{D}(G)$ by Jaikin-Zapirain's appendix to \cite{JaikinZapirain2020THEUO}. Then, the $\ell^2$-Betti numbers of $G$ are defined and $b_p^{(2)}(G) = b_p^{\mathcal{D}_{\Q G}}(G)$. Hence, applying \cref{thm:agrarianMain} to the case $\mathbb{F} = \Q$ yields the following result stated in the introduction.

\begin{thm}\label{thm:b2rfrs}
Let $G$ be a virtually RFRS group of type $\mathtt{FP}_n(\Q)$. Then there is a finite index subgroup $H \leqslant G$ admitting a homomorphism onto $\Z$ with kernel of type $\mathtt{FP}_n(\Q)$ if and only if $b_p^{(2)}(G) = 0$ for $p = 0, \dots, n$.
\end{thm}

Thanks to Jaikin-Zapirain's work on rank functions in \cite{JaikinZapirain2020THEUO}, we can give a third characterisation of algebraic fibring. First, we set up some notation and terminology. If $R$ is a ring and $\varphi \colon R \rightarrow \mathcal D$ is a division $R$-ring, then there is a natural rank function on matrices over $R$, which we denote $\rk_{\mathcal D, \varphi}$ or simply $\rk_{\mathcal D}$ when the map $\varphi$ is understood. If $\varphi$ is \textit{epic}, meaning that $\varphi(R)$ generates $\mathcal D$ as a division ring, then we call $\mathcal D$ \textit{universal} if the $\rk_\mathcal D \geqslant \rk_\mathcal E$ for every division $R$-ring $\psi \colon R \rightarrow \mathcal E$. Note that if a universal division $R$-ring exists, then it is unique up to $R$-isomorphism by a result of Cohn \cite[Theorem 4.4.1]{cohn1995skew}. If, additionally, $\varphi$ is an injection, then we call $\mathcal D$ the \textit{universal division ring of fractions} for $R$.

\begin{thm}[Jaikin-Zapirain, {\cite[Corollary 1.3]{JaikinZapirain2020THEUO}}]  \label{thm:jaikinUniv}
    If $G$ is a residually-(locally indicable and amenable) group and $\mathbb F$ is a skew-field, then the Hughes-free division ring $\DF{G}$ exists and is the universal division ring of fractions for $\mathbb FG$.
\end{thm}

Following Jaikin-Zapirain, whenever $G$ has a Hughes-free division ring $\DF{G}$, we denote the rank function $\rk_{\DF{G}}$ by $\rk_{\mathbb FG}$. Note, in particular, that \cref{thm:jaikinUniv} holds for RFRS groups since they are residually poly-$\Z$ (see, e.g., \cite[Proposition 4.4]{JaikinZapirain2020THEUO}). Therefore, we have the following corollary, which will be useful to us below.

\begin{cor}\label{cor:rkIneq}
    Let $G$ be a RFRS group and let $\varphi \colon G \rightarrow \Z$ be a homomorphism. Let $A$ be a matrix over $\mathbb FG$ and let $A^\Z$ be the matrix over $\mathbb F[\Z]$ obtained by applying $\varphi$ to $A$. Then $\rk_{\mathbb FG} A \geqslant \rk_{\mathbb F[\Z]} A^\Z$.
\end{cor}

\begin{proof}
    There is a map $\overline{\varphi} \colon \mathbb F G \rightarrow \mathbb F[\Z] \hookrightarrow \DF{[\Z]}$ induced by $\varphi$ and therefore $\rk_{\mathbb FG} A \geqslant \rk_{\DF{[\Z]}, \overline{\varphi}} A = \rk_{\mathbb F[\Z]} A^\Z$ by universality of $\DF{G}$. \qedhere
\end{proof}

The following theorem generalises Corollary 1.5 of \cite{JaikinZapirain2020THEUO}, where the result is proven for $n = 1$. In the proof, if $M$ is a finitely generated $\mathbb F[\Z]$-module, then we define $\dim M := \dim_{\DF{[\Z]}} (M \otimes_{\mathbb F[\Z]} \DF{[\Z]})$. We will also write $\rk_G$ instead of $\rk_{\mathbb FG}$ to lighten the notation.

\begin{thm}\label{thm:finiteBetti}
    Let $\mathbb F$ be a skew-field and let $G$ be a virtually RFRS group of type $\FP_n(\mathbb F)$. The following are equivalent:
    \begin{enumerate}[label=(\arabic*)]
        \item\label{item:FP} there is a finite index subgroup $H_0 \leqslant G$ and a surjection $\varphi_0 \colon H_0 \rightarrow \Z$ with $\ker \varphi_0$ of type $\FP_n(\mathbb F)$;
        \item\label{item:Betti} there is a finite index subgroup $H_1 \leqslant G$ and a surjection $\varphi_1 \colon H_1 \rightarrow \Z$ with $b_p(\ker \varphi_1; \mathbb F) < \infty$ for $p = 0, \dots, n$.
    \end{enumerate}
\end{thm}

\begin{proof}
    If $G$ algebraically fibres with kernel $K$ of type $\FP_n(\mathbb F)$, then there is a free resolution $F_\bullet \rightarrow \mathbb F \rightarrow 0$ of the trivial $\mathbb F K$ module $\mathbb F$ such that $F_p$ is finitely generated for all $p \leqslant n$. This resolution can be used to compute the homology of $K$, and therefore $b_p(K; \mathbb F) < \infty$ for all $p \leqslant n$.
    
    In view of \cref{thm:agrarianMain}, to prove the converse it suffices to show that $b_p^{\DF{G}}(G) = 0$ for all $p \leqslant n$. By multiplicativity of $\DF{G}$-Betti numbers (\cref{lem:JBscales}), we may assume that $H_1 = G$. Let $K = \ker \varphi_1$ and write $\mathbb F[\Z]$ for the group algebra $\mathbb F[G/K]$. Moreover, note that $H_p(K;\mathbb F) \cong H_p(G; \mathbb F[\Z])$ for all $p$. Let
    \[
        \cdots \rightarrow \mathbb F G^{d_p} \rightarrow \cdots \rightarrow \mathbb F G^{d_0} \rightarrow \mathbb F \rightarrow 0
    \]
    be a free resolution of the trivial $\mathbb F G$-module $\mathbb F$, where $d_p$ is some cardinal for each $p$ and $d_p$ is finite for each $p \leqslant n$, and we use the (non-standard) notation $\mathbb F G^{d_p}$ to denote the $d_p$-fold direct sum of $\mathbb F G$'s, as opposed to the $d_p$-fold direct product. The quotient map $G \rightarrow \Z$ induces a chain map
    \[
        \begin{tikzcd}
        \cdots \arrow[r, "\partial_{n+2}"]    & \mathbb F G^{d_{n+1}} \arrow[r, "\partial_{n+1}"] \arrow[d] & \mathbb F G^{d_n} \arrow[r, "\partial_n"] \arrow[d] & \cdots \arrow[r, "\partial_1"]    & \mathbb F G^{d_0} \arrow[r, "\partial_0"] \arrow[d] & 0 \\
        \cdots \arrow[r, "\partial_{n+2}^\Z"] & {\mathbb F[\Z]^{d_{n+1}}} \arrow[r, "\partial_{n+1}^\Z"]        & \mathbb F[\Z]^{d_n} \arrow[r, "\partial_n^\Z"]    & \cdots \arrow[r, "\partial_1^\Z"] & {\mathbb F[\Z]^{d_0}} \arrow[r, "\partial_0^\Z"]    & 0\nospacepunct{,}
    \end{tikzcd}
    \]
    where the boundary maps are viewed as matrices and $\partial_p^\Z$ is obtained by applying the map $G \rightarrow \Z$ to each entry of the matrix $\partial_p$. Note that the homology of the bottom chain complex is $H_\bullet(G;\Q[\Z])$.
    
    To apply Jaikin-Zapirain's results on rank functions, we will need the boundary maps to be between finitely generated free modules. However, $d_{n+1}$ is not finite in general, so we must modify the chain complexes as follows. Since $\mathbb F[\Z]$ is Noetherian and $\mathbb F[\Z]^{d_n}$ is finitely generated, $\im \partial_{n+1}^\Z$ is a finitely generated submodule of $\mathbb F[\Z]^{d_n}$. The preimage of a finite generating set of $\im \partial_{n+1}^\Z$ is contained in a finitely generated free summand $F$ of $\mathbb F[\Z]^{d_{n+1}}$. Notice that the homology of 
    \[    
        F \rightarrow \mathbb F[\Z]^{d_n} \rightarrow \cdots \rightarrow \mathbb F[\Z]^{d_0} \rightarrow 0
    \]
    is still $H_p(G; \mathbb F[\Z])$ for $p \leqslant n$. The preimage of $F$ in $\mathbb F G^{d_{n+1}}$ is again a finitely generated free summand $\widehat{F}$ of $\mathbb F G^{d_{n+1}}$. Note that it suffices to show that the homology of
    \[
        \DF{G} \otimes_{\Q G} \widehat{F} \rightarrow \DF{G} \otimes_{\mathbb F G} \mathbb F G^{d_n} \rightarrow \cdots \rightarrow \DF{G} \otimes_{\mathbb F G} \mathbb F G^{d_0} \rightarrow 0
    \]
    vanishes in degrees $\leqslant n$ to show that $b_p^{\DF{G}}(G) = 0$ for all $p \leqslant n$.

    We assume that $d_{n+1}$ is finite and that $F = \mathbb F[\Z]^{d_{n+1}}$ and $\widehat{F} = \mathbb F G^{d_{n+1}}$. Since, for every $p \leqslant n$, the homology $H_p(G;\mathbb F[\Z])$ is finite-dimensional as an $\mathbb F$-vector space, it must be torsion as an $\mathbb F[\Z]$-module. Therefore, $\rk_\Z \partial_{p+1}^\Z = \dim \ker \partial_p^\Z$ for every $p \leqslant n$. Now, for each $p \leqslant n$, we have short exact sequences
    \[
        0 \rightarrow \ker \partial_p^\Z \rightarrow \mathbb F G^{d_p} \rightarrow \im \partial_p^\Z \rightarrow 0
    \]
    which implies that $d_p = \dim \ker \partial_p^\Z + \rk_\Z \partial_p^\Z = \rk_\Z \partial_{p+1}^\Z + \rk_\Z \partial_p^\Z$. Hence,
    \begin{align*}
        d_p - \rk_G \partial_p &= \dim_{\DF{G}} \ker(\DF{G}^{d_p} \xrightarrow{\partial_p} \DF{G}^{d_{p-1}} ) \\
        &\geqslant \rk_G \partial_{p+1} \\
        &\geqslant \rk_\Z \partial_{p+1}^\Z \\
        &= d_p - \rk_\Z \partial_p^\Z \\
        &\geqslant d_p - \rk_G \partial_p,
    \end{align*}
    where we have used \cref{cor:rkIneq}. Thus, 
    \[
        \rk_G \partial_{p+1} = \dim_{\DF{G}} \ker(\DF{G}^{d_p} \xrightarrow{\partial_p} \DF{G}^{d_{p-1}} ),
    \]
    and therefore $b_p^{\DF{G}}(G) = 0$ for all $p \leqslant n$. \qedhere
\end{proof}

\begin{cor}\label{cor:charac}
    Let $G$ be a virtually RFRS group and let $n \in \N$.
    \begin{enumerate}[label=(\arabic*)]
        \item\label{item:same_char} If $\mathbb F$ and $\mathbb F'$ are skew-fields of the same characteristic, then $G$ virtually algebraically fibres with kernel of type $\FP_n(\mathbb F)$ if and only if it virtually algebraically fibres with kernel of type $\FP_n(\mathbb F')$.
        \item\label{item:p_vs_q} If $p$ is a prime such that $G$ algebraically fibres with kernel of type $\FP_n(\mathbb F_p)$, then it fibres with kernel of type $\FP_n(\Q)$.
    \end{enumerate}
\end{cor}

\begin{proof}
    \cref{item:same_char} follows from the fact that the Betti numbers of a group with trivial skew-field coefficients depend only on the characteristic of the skew-field. \cref{item:p_vs_q} follows from the fact that $b_k(G; \mathbb F_p) \geqslant b_k(G; \Q)$ for any group $G$ and any prime $p$ (this is a consequence of the universal coefficient theorem). \qedhere
\end{proof}

\section{Applications} \label{sec:app}

\subsection{Amenable RFRS groups}

\begin{defn}
A group $G$ is \textit{amenable} if for every continuous $G$-action on a compact, Hausdorff space $X$, there is a $G$-invariant probability measure on $X$.
\end{defn}

\begin{defn}[elementary amenable groups] \label{def:elemAm}
The class $\mathcal{E}$ of \textit{elementary amenable} groups is the smallest class such that 
\begin{enumerate}[label = {$\bullet$}]
    \item $\mathcal{E}$ contains all finite groups and all Abelian groups;
    \item if $G \in \mathcal{E}$ then the entire isomorphism class of $G$ is contained in $\mathcal{E}$;
    \item $\mathcal{E}$ is closed under taking subgroups, quotients, extensions, and directed unions.
\end{enumerate}
\end{defn}

All elementary amenable groups are amenable, however there are many examples of amenable groups that are not elementary amenable, the earliest being Grigorchuk's group of intermediate growth \cite{GrigorchukGroup}. For a discussion of more examples, we refer the reader to the introduction of Juschenko's paper \cite{JuschenckoNEA}. The known examples of non-elementary amenable groups all have infinite cohomological dimension over any field.  Moreover, elementary amenable groups of finite cohomological dimension over $\Z$ are virtually solvable by \cite[Lemma 2]{Hillman91} and \cite[Corollary 1]{HillmanLinnell}. We are led to the following question, which was stated in the introduction. 

\begin{q}
Are amenable groups of finite cohomological dimension over $\Z$ virtually solvable?
\end{q}

We obtain a partial answer in the positive direction as an application of \cref{thm:b2rfrs}, and extends the well-known fact that nilpotent RFRS groups are Abelian. The author thanks Sami Douba for pointing out the fact (and the proof) that polycyclic RFRS groups are virtually Abelian, which allows us to strengthen the following theorem.

\begin{thm}\label{thm:amRFRSelemAm}
    If $G$ is an amenable RFRS group of type $\mathtt{FP}(\Q)$, then $G$ is virtually Abelian.
\end{thm}
\begin{proof}
    Since $G$ is amenable, $b_p^{\mathcal D_{\Q G}}(G) = 0$ for all $p$ by \cite[Theorem 7.2(1)]{Luck02}. By \cref{thm:b2rfrs} we obtain a finite index subgroup $H \leqslant G$ and an epimorphism $\varphi \colon H \rightarrow \Z$ such that $N = \ker \varphi \in \mathtt{FP}_n(\Q)$. It will be necessary to require that $H$ be normal in $G$, which is not an issue since we can replace $H$ with its normal core. Since $\cd_{\Q} N \leqslant \cd_\Q G \leqslant n$, we also have $N \in \mathtt{FP}(\Q)$ \cite[Proposition VIII.6.1]{BrownGroupCohomology}. Because we have a short exact sequence
    \[
        1 \rightarrow N \rightarrow H \rightarrow \Z \rightarrow 1
    \]
    with $N \in \mathtt{FP}(\Q)$, a theorem of Fel'dman \cite[Theorem 2.4]{Feldman71} (see also \cite[Proposition 2.5]{Bieri76}) gives $\cd_\Q N = \cd_\Q H - \cd_\Q \Z = n - 1$.

    Since subgroups of amenable RFRS groups are amenable and RFRS, we can repeat the argument above with $N$ instead of $G$. Iterating this process, we obtain a subnormal series
    \[
        G_0 \trianglelefteqslant G_1 \trianglelefteqslant \cdots \trianglelefteqslant G_{n-1} \trianglelefteqslant G_n = G
    \]
    of $G$ such that $\cd_\Q G_j = j$ for each $j$ (note that $N = G_{n-1}$ here). But the only torsion-free group of cohomological dimension $0$ over $\Q$ is the trivial group, so we conclude that $G$ is polycyclic-by-finite.

    Let $G' \leqslant G$ be a polycyclic group of finite index. We will show, by induction on the Hirsch length of $G'$, that $G'$ must be virtually Abelian. Indeed, by induction we may assume that $G' \cong \Z^m \rtimes_\psi \Z$. If $\psi$ is of finite order, then $G'$ is virtually Abelian. If $\psi$ is of infinite order, then the argument of \cite[Proposition 4.17]{BridsonHaefliger_thebook} shows that $\Z^m$ cannot be quasi-isometrically embedded in $G'$. However, \cref{prop:RFRStoQ}\ref{item:v_retracts} implies that all Abelian subgroups of a finitely generated RFRS group are undistorted, a contradiction. \qedhere
\end{proof}

\begin{rem}
    The assumption that $G$ be RFRS is necessary. For example, the Baumslag--Solitar group $\BS(1,n)$, with $n > 1$, is locally indicable, solvable, and of finite type, but it is not virtually Abelian (nor is it polycyclic-by-finite). The assumption that $G$ be (virtually) of type $\FP(\Q)$ is also necessary; for example $\Z \wr \Z$ is RFRS and amenable, yet it is not virtually Abelian.
\end{rem}

Baer's conjecture states that if $G$ is a group with a Noetherian group ring $\Z G$, then $G$ is a polycyclic-by-finite group. The author is grateful to Sam Hughes for pointing out the following consequence of \cref{thm:amRFRSelemAm}, which resolves a special case of Baer's conjecture.

\begin{cor}\label{cor:baer}
    Let $G$ be a virtually RFRS group of type $\mathtt{FP}(\Q)$. If $\Z G$ is Noetherian, then $G$ is virtually Abelian.
\end{cor}
\begin{proof}
    It is enough to prove the claim in the case that $G$ is RFRS. In this case, $\mathbb{Z} G$ embeds into the Linnell skew-field $\mathcal{D}(G)$, and therefore $\Z G$ is a domain. Since Noetherian domains are Ore domains \cite[p.~47]{McConnellRobsonNNR}, we have that $\Z G$ is an Ore domain. It then follows that $\Q G$ is an Ore domain, which implies that $G$ is amenable by Kielak's appendix to \cite{BartholdiKielakApp}. Then $G$ is polycyclic-by-finite by \cref{thm:amRFRSelemAm}.
\end{proof}

\subsection{Arithmetic lattices and subgroups of hyperbolic groups}

Recall that hyperbolic groups are always of type $\F_\infty$ (and of type $\F$ if they are torsion-free). However, their subgroups can exhibit many interesting finiteness properties. In \cite{RipsF1notF2_1982}, Rips gave the first example of an incoherent hyperbolic group; phrased in terms of finiteness properties, this gives an example of a hyperbolic group with a subgroup that is of type $\F_1$ but not $\F_2$. In \cite{BradyF2notF3_1999}, Brady constructed a type $\F_2$ subgroup of a hyperbolic group that is not of type $\F_3$; this provided the first example of a finitely presented non-hyperbolic subgroup of a hyperbolic group. In the same paper, Brady asked whether there are subgroups of hyperbolic groups that are of type $\F_n$ but not $\F_{n+1}$ for all $n$. More examples of $\F_2$-not-$\F_3$ subgroups of hyperbolic groups were provided by Kropholler \cite{KrophollerF2notF3} and Lodha \cite{LodhaF_2notF_3}, and in \cite{IsenrichPyMartelli_F3notF4}, Llosa Isenrich, Martelli, and Py constructed the first examples of $\F_3$-not-$\F_4$ subgroups of hyperbolic groups. This result was extended in a subsequent paper \cite{IsenrichPy2022}, where Llosa Isenrich and Py completely answer Brady's question by exhibiting cocompact hyperbolic arithmetic lattices in $\PU(n,1)$ with $\F_{n-1}$-not-$\F_n$ subgroups for all $n$. We also mention the related work of Italiano, Martelli, and Migliorini, who constructed the first example of a non-hyperbolic type $\F$ subgroup of a hyperbolic group \cite{IMM_5mfld}.

In \cite{IsenrichPyMartelli_F3notF4}, Llosa Isenrich, Martelli, and Py remark that one can use \cref{thm:b2rfrs} to show that there are subgroups of hyperbolic lattices in $\PO(2n,1)$ that are of type $\FP_{n-1}(\Q)$ but not $\FP_n(\Q)$. We record their argument here, and also mention that the same line of reasoning shows that there are many hyperbolic lattices in $\PO(2n+1,1)$ that virtually fibre with kernel of type $\FP(\Q)$.

\begin{prop}[{\cite[Proposition 19]{IsenrichPyMartelli_F3notF4}}]\label{prop:lattices}
    Let $\Gamma < \PO(n,1)$ be a cocompact cubulable lattice. If $n = 2k$ is even then $\Gamma$ virtually fibres with kernel of type $\FP_{k-1}(\Q)$ but not $\FP_k(\Q)$. If $n$ is odd, then $\Gamma$ virtually fibres with kernel of type $\FP(\Q)$.
\end{prop}

\begin{rem}
    In \cite{BHW_2011}, Bergeron, Haglund, and Wise show that any standard arithmetic subgroup of $\PO(n,1)$ is cubulated, so \cref{prop:lattices} applies to a nonempty class of groups. We refer the reader to their paper for the definition of standard.
\end{rem}

\begin{proof}
    By Agol's theorem \cite[Theorem 1.1]{AgolHaken}, $\Gamma$ is virtually special and in particular virtually RFRS. By, for instance, \cite[Theorem 3.3]{KammeyerLattices}, the $\ell^2$-Betti numbers of lattices in semisimple Lie groups vanish except in the middle dimension, where they are nonzero. If $n = 2k$, then $\Gamma$ virtually fibres with kernel of type $\FP_{k-1}(\Q)$, but $b_k^{(2)}(\Gamma) \neq 0$ so the kernel cannot be of type $\FP_n(\Q)$ by \cref{thm:agrarianMain}. If $n$ is odd, then $\Gamma$ is $\ell^2$-acyclic and therefore virtually fibres with kernel of type $\FP(\Q)$ in this case, where we have also used \cref{cor:typeFP}. \qedhere
\end{proof}

\section{Addendum: Baer's conjecture for locally indicable groups}

In \cref{cor:baer}, we showed that Baer's Conjecture holds for RFRS groups of type \(\FP(\Q)\). These assumptions are unnecessarily strong. Indeed, using elementary arguments, we show that Baer's Conjecture holds for virtually locally indicable groups. All modules are left modules by convention, and we say that a ring is Noetherian if all of its left ideals are finitely generated.

\begin{lem}
    Let \(G\) be a group and \(R\) be a ring such that \(RG\) is Noetherian. Then \(RH\) is Noetherian for every subgroup \(H \leqslant G\).
\end{lem}
\begin{proof}
    Let \(I \subseteq RH\) be an ideal and let \(I^G\) be the ideal of \(RG\) generated by \(I\). Then \(I^G\) is generated by a finite set \(\{x_1, \dots, x_n\} \subseteq RH\), since \(RG\) is Noetherian. We claim that \(\{x_1, \dots, x_n\}\) generates \(I\). Indeed, any \(y \in I\) can be written as
    \[
        y = a_1 x_1 + \dots + a_n x_n
    \]
    for some elements \(a_i \in RG\). For each \(i = 1, \dots, n\), write \(a_i = b_i + c_i\), where \(b_i \in RH\), and the support of \(c_i\) is contained in \(G \smallsetminus H\). Then the equation
    \[
        y = (b_1 x_1 + \dots + b_n x_n) + (c_1x_1 + \dots + c_nx_n)
    \]
    implies that \(c_1x_1 + \dots + c_nx_n = 0\), and therefore that \(I\) is generated by the elements \(x_1, \dots, x_n\). \qedhere
\end{proof}

\begin{lem}
    Let \(G\) be a group and \(R\) be a ring such that \(RG\) is Noetherian. Then \(G\) is of type \(\FP_\infty(R)\).
\end{lem}
\begin{proof}
    We construct a resolution of the trivial \(RG\)-module \(R\) by induction. The augmentation ideal, i.e.|the kernel of the map \(RG \rightarrow R\) sending all the group elements to \(1 \in R\), is finitely generated, so there is a partial resolution \(RG^{n_1} \rightarrow RG \rightarrow R \rightarrow 0\) for some integer \(n_1\). Now suppose that we have constructed a partial resolution
    \[
        RG^{n_i} \rightarrow RG^{n_{i-1}} \rightarrow \dots \rightarrow RG^{n_1} \rightarrow RG \rightarrow R \rightarrow 0.
    \]
    Since \(RG^{n_i}\) is finitely generated and \(RG\) is a Noetherian ring, it is a Noetherian \(RG\)-module. Hence, the kernel of \(RG^{n_i} \rightarrow RG^{n_{i-1}}\) is finitely generated, and we can continue the resolution with \(RG^{n_{i+1}} \rightarrow RG^{n_i} \rightarrow RG^{n_{i-1}}\) for some integer \(n_{i+1}\). \qedhere
\end{proof}

\begin{cor}
    Let \(G\) be a locally indicable group with \(\cd(G) < \infty\). If \(\Z G\) is Noetherian, then \(G\) is poly-\(\Z\).
\end{cor}
\begin{proof}
    We prove the claim by induction on \(\cd(G)\). If \(\cd(G) = 0\), then \(G = \{1\}\) is trivially poly-\(\Z\). Now suppose that \(\cd(G) > 0\) and let \(1 \rightarrow H \rightarrow G \rightarrow \Z \rightarrow 1\) be a short exact sequence (which exists because \(G\) is locally indicable). By the previous two lemmas, all the groups appearing in the short exact sequence are of type \(\FP\). By \cite[Theorem 2.4]{Feldman71}, \(\cd(H) < \cd(G)\), so \(H\) is poly-\(\Z\) by induction. Hence, \(G\) is poly-\(\Z\). \qedhere
\end{proof}

\bibliographystyle{alpha}
\bibliography{Fisher}

\end{document}

%% file: top.tex
\usepackage{amsmath}
\usepackage{amsthm}
\usepackage{amssymb}
\usepackage{bm}
\usepackage{dsfont}
\usepackage{enumitem}
\usepackage{graphicx}
\usepackage[hidelinks]{hyperref}
\usepackage[capitalise]{cleveref}
\usepackage{mathrsfs}
\usepackage{mathtools}
\usepackage{nameref}
\usepackage[new]{old-arrows}
\usepackage{stmaryrd}
\usepackage{thmtools}
\usepackage{xcolor}
\usepackage{xparse}

\usepackage{tikz}
\usetikzlibrary{cd}
\usetikzlibrary{positioning, angles, quotes}
\newcommand{\nospacepunct}[1]{\makebox[0pt][l]{\,#1}} 

\theoremstyle{plain}
\newtheorem{thm}{Theorem}[section]
\newtheorem*{thm*}{Theorem}

\newtheorem{cor}[thm]{Corollary}
\newtheorem*{cor*}{Corollary}

\newtheorem{prop}[thm]{Proposition}
\newtheorem*{prop*}{Proposition}

\newtheorem{lem}[thm]{Lemma}
\newtheorem*{lem*}{Lemma}

\newtheorem*{ex*}{Example}

\newtheorem*{exer*}{Exercise}

\newtheorem{q}[thm]{Question}
\newtheorem*{q*}{Question}

\newtheorem{conj}[thm]{Conjecture}
\newtheorem*{conj*}{Conjecture}

\theoremstyle{definition}
\newtheorem{defn}[thm]{Definition}
\newtheorem*{defn*}{Definition}

\theoremstyle{remark}
\newtheorem{rem}[thm]{Remark}
\newtheorem*{rem*}{Remark}

\theoremstyle{plain}
\newtheorem{manualtheoreminner}{Theorem}
\newenvironment{manualtheorem}[1]{%
  \renewcommand\themanualtheoreminner{#1}%
  \manualtheoreminner
}{\endmanualtheoreminner}

\newtheorem{manualcorinner}{Corollary}
\newenvironment{manualcor}[1]{%
  \renewcommand\themanualcorinner{#1}%
  \manualcorinner
}{\endmanualcorinner}

\Crefname{defn}{Definition}{Definitions}

\newcommand{\BS}{\mathrm{BS}}

\newcommand{\PO}{\mathrm{PO}}
\newcommand{\PU}{\mathrm{PU}}

\newcommand{\F}{\mathtt{F}}
\newcommand{\FP}{\mathtt{FP}}

\newcommand{\DF}[1]{\mathcal{D}_{\mathbb F{#1}}}

\newcommand{\N}{\mathbb{N}}
\newcommand{\Q}{\mathbb{Q}}
\newcommand{\R}{\mathbb{R}}
\newcommand{\Z}{\mathbb{Z}}

\newcommand{\inv}{^{-1}}

\DeclareMathOperator{\cd}{cd}

\DeclareMathOperator{\id}{id}
\DeclareMathOperator{\im}{im}

\DeclareMathOperator{\Ore}{Ore}
\DeclareMathOperator{\rk}{rk}
\DeclareMathOperator{\supp}{supp}
\DeclareMathOperator{\Tor}{Tor}

\newcommand{\verti}[1]{{\left\vert #1 \right\vert}}
\newcommand{\vertii}[1]{{\left\vert\kern-0.25ex\left\vert #1 \right\vert\kern-0.25ex\right\vert}}
\newcommand{\vertiii}[1]{{\left\vert\kern-0.25ex\left\vert\kern-0.25ex\left\vert #1 \right\vert\kern-0.25ex\right\vert\kern-0.25ex\right\vert}}